\documentclass[12pt,a4paper]{article}
\usepackage{amsmath, amssymb}

\advance\topmargin by -2.2cm
\newtheorem{theo}{Theorem}[section]
\newtheorem{prop}[theo]{Proposition}

\newtheorem{rem}[theo]{Remark}
\newtheorem{lemma}[theo]{Lemma}
\newtheorem{definition}[theo]{Definition}
\newtheorem{cor}[theo]{Corollary}

\newtheorem{example}[theo]{Example}
\newtheorem{ass}[theo]{Assumption}
\numberwithin{equation}{section}

\newenvironment{proof}{\noindent {\bf Proof }}
{\hfill $\bullet$ \vspace{0.25cm}}

\newcommand{\lat}{\mathbb{Z}^d}

\def\comp#1{{#1}^{\rm c}}

\def\buildd#1#2{\mathrel{\mathop{\kern 0pt#1}\limits_{#2}}}

\newcommand{\supt}{\rm sp}

\def\sobre#1#2{\lower 1ex \hbox{ $#1 \atop #2 $ } }

\def\FF{{\mathcal F}}

\def\FF{{\mathcal F}}

\def\E{{\mathbb E}}
\def\P{{\mathbb P}}
\def\R{{\mathbb R}}
\def\Z{{\mathbb Z}}

\def\N{{\mathbb N}}

\def\cG{{\cal G}}

\def\cA{{\cal A}}

\def\cG{{\cal G}}
\def \e {{\epsilon}}

\def\ll{{\bar \Lambda_n}}
\def\sqr{\vcenter{
         \hrule height.1mm
         \hbox{\vrule width.1mm height2.2mm\kern2.18mm\vrule width.1mm}
         \hrule height.1mm}}                  % This is a slimmer sqr.
\def\square{\ifmmode\sqr\else{$\sqr$}\fi}

\def\cR{{\cal R}}

\def \e {{\epsilon}}
\def \g {{\gamma}}
\def \om {{\omega}}
\def\L{{\Lambda}}
\def\G{{\Gamma}}

\title{Neighborhood radius estimation  for Variable-neighborhood   Random Fields} 
\date{}
\author{ Eva L\"ocherbach, Enza Orlandi}
\date{February 16, 2011}

\begin{document}
\maketitle

\abstract We consider  random fields defined by finite-region
conditional probabilities depending on a neighborhood of the region
which changes with the boundary conditions.  
To  predict the symbols within  any finite region  
  it is necessary to   inspect a  random  number of neighborhood
symbols which might change according to the value of them. In analogy
to the one dimensional setting we call these neighborhood symbols   the {\it context} associated
to the region at hand.
This framework is a natural extension, to $d$-dimensional fields, of
the notion of variable-length Markov chains introduced by Rissanen
(1983) in his classical paper. We define an algorithm to estimate the
radius of the smallest ball containing the context based on a
realization of the field.  We prove the consistency of this estimator. 
 {Our proofs are constructive and yield explicit
upper bounds for the probability of wrong estimation of the radius of
the context.}



\endabstract

\section{Introduction}
We consider random fields on $\Z^d$ with finite
state space %$\cA$, 
defined by prescribing a family of conditional
probabilities indexed by finite subsets $\Lambda$ of $\Z^d .$
% It is a
%natural assumption 
We assume that these conditional probabilities depend on a
finite neighborhood which changes according to the boundary conditions.
Contrary to standard Markov random fields which are defined by a
family of conditional probabilities depending on a {\it fixed }
neighborhood and not sensitive to the boundary conditions (fixed order
Markov dependence), the families of conditional probabilities
considered here are not restricted to a predefined uniform depth.
Rather, by examining the training data, a model is constructed that
fits higher order Markov dependencies where needed, while using lower
order Markov dependence elsewhere. We call these random fields {\it
  Variable-neighborhood random fields} or {\it Parsimonious Markov
  random fields}.

Adopting this parsimonious description means that we are aiming at
reducing information by finding the minimal neighborhood of a given
block of sites  able to predict the  states of the  sites within this block.
The neighborhood changes when the outside configuration of the field
changes, and the dependencies depend on the realization of the field.

 Applications of  Variable-neighborhood  random fields 
are in image analysis: in texture synthesis, computer 
vision and graphics. We
refer the interested reader to Efros and Leung (1999), \cite{EL}, for
the presentation of a non-parametric texture synthesis method. Texture is usually modelled as Markov random field,   
``composed of well defined texture primitives (texels) which are placed according to
some syntactic rules'' (Gidofalvi (2001), \cite{Gido}). Thus
a modelization as Variable-neighborhood random field where the width of the context
window may depend on the realization of the field is natural. 
Other possible applications of Variable-neighborhood random fields
are in neuroscience's and in general  in  spatial statistics, whenever  information reduction is needed.

The notion of Variable-neighborhood  random fields  has been inspired by Rissanen's Minimum Description
Length principle for Markov chains, see Rissanen (1983), \cite {R}.
Rissanen calls the relevant neighborhood of a site, i.e.~the sequence
of symbols needed to predict the next symbol, given a finite sample,
{\it context} of a site and proposes an estimator of the length of the
context. Since this seminal paper, there have been several
implementations and extensions of the method. We refer  to the book of Grunwald (2007),  
\cite {Gr}, and to a review
paper by Galves and L\"ocherbach (2008), \cite {GL},    for a   comprehensive introduction. 
 Results for the context algorithm can be found  for example   in Ferrari and Wyner (2003)  and  in  Galves, Maume-Deschamps and Schmitt  (2008), \cite {GMS}.    In this last paper, the rate of convergence of the context algorithm is established and non asymptotic error bounds implying the consistency of the estimator are obtained. All these results  are related to processes in dimension
one.  Our aim is to extend this method to more than one dimension and to define an estimator of
the context in the framework of random fields.

This requires to define a random field which can predict the symbol at a given
site by inspecting a ``random'' number of neighborhood
symbols which might change according to the value of them. In analogy
to the one dimensional setting we call this neighborhood, i.e.~the
subset of symbols needed to predict the symbol at the given site, the {\it context} of this site. %We define the Variable-neighborhood random field through a
%family of finite-region conditional probabilities.
 For such random fields we estimate the
radius of the context of a given site, i.e.~the radius of the smallest ball
containing the context of this site. It is enough to consider the
contexts for one site, since in our setting %, by Theorem 1.33 of
%Georgii (1988), \cite {geo88},  
the one point specification uniquely determines the specification for any other set.  We apply a penalized
pseudo-likelihood method, first introduced by Besag (1975), \cite
{Be}, and developed by Csisz\`ar and Talata (2006), \cite{CT}, in
order to construct our estimator. Our estimator is a function of the
observed blocks or patterns appearing in the sample. It is based
on a sequence of local decisions between two possible values of the radius of the context, lumping them together whenever their corresponding one point conditional probabilities are similar. 
We propose an
estimator for any site within our observation window, depending on
its local neighborhood. Hence we deal with a family of estimators
indexed by the centers of observation patterns. For this family of
estimators, we give in Theorem \ref{theomain1} and Theorem \ref{theomain2} explicit  error-bounds for
the probability of over- and underestimation. These bounds are non
asymptotic with respect to the number of observed sites, i.e.~the size
of the observation window. As a consequence, we obtain the consistency
of the neighborhood radius estimator.
   
Our results are based on several deviation inequalities which are
interesting in its own right. They are collected in Sections \ref{section4} and \ref{section:5}. The
first part of them  (Section \ref{section4}) is based on results obtained by Dedecker (2001), \cite{dedecker}, on
deviation inequalities for random fields, the second part (Section \ref{section:5})  is a
rewriting of typicality results obtained by Csisz\`ar and Talata
(2006), \cite{CT}.   Csisz\`ar and Talata are only interested in
consistency and they do not give explicit upper bounds for the error probabilities.
 We want to control the error bounds, for any
fixed $n,$ and so we carry on their ideas into non-asymptotic
deviation inequalities.

We implement the estimates under the  requirement  that the one point conditional probabilities are strictly positive.  This is enough for the overestimation. 
To implement the estimates for the underestimation, we need to assume that Dobrushin's contraction condition holds, see  Dobrushin  (1968), \cite {D1} and \cite {D2},   and that  there
exists some finite order $L,$ unknown to the statistician, such that the random field is
Markov of order at most $L.$ In the language of context-trees this
means that we deal with finite trees only.

%Concerning the history of statistics for Random Fields, 
There is large number of papers devoted to parameter estimation for
Markov random fields when the structure of the interaction is known,
see for example Gidas (1993), \cite {gidas},  Comets (1992),  \cite {comets}, 
 Dereudre and Lavancier (2011), \cite{DL},  and many others.
Typically, the parameter estimation addresses the problem of
estimating parameters entering in determining the potential, but not
directly the conditional probabilities.  Quite recently, the
non-parametric problem of model selection has been addressed, i.e.~the
statistical estimation of the interaction structure, see for example
Ji and Seymour (1996), \cite {JS}. Csisz\`ar and Talata (2006),
\cite{CT}, propose to estimate the basic neighborhood of Markov random
fields and estimate the support of the neighborhood (i.e.~its
geometrical structure) which is relevant to determine the conditional
probabilities. In their framework this neighborhood does not depend on
the configuration, hence they work in a strict Markovian frame. 
In Galves et al. (2010), \cite {GOT},  a related problem has been studied.
The authors estimate
 for an Ising model having pairwise interactions of infinite range the pairs of interacting sites based on i.i.d. observations of the field.   Our
paper is not situated in the same framework. We do not address the
problem of estimating the geometrical structure of the contexts,
since this would require to introduce too many free parameters. We
deal with a problem which is simpler and more difficult at the same
time: we estimate only the radius of the basic neighborhood, but this
neighborhood varies when the configuration changes.  This last feature is the main difference from previous models which have appeared in the literature.

The paper is organized as follows.  In Section 2  we define the Variable-neighborhood  random fields, 
based on the prescription of a ``variable-neighborhood''-specification 
 and we provide examples.  In Section 3 we define the estimator of the radius of a single-site context
and formulate the main results.  In Theorem \ref{theomain1} 
we  give 
the   bound  on the probability of overestimation and in Theorem \ref{theomain2} the  bound  on the probability of   underestimation,
under suitable assumptions on the decay of correlations in the field.
In Section 4  we prove the deviation inequalities 
needed for   controlling 
 the underestimation   and in   Section 5  those needed for controlling the overestimation. 
In Sections 6 and 7 we
give the proof of the main results. We conclude with some final remarks in Section 8.  In Section 9,  the appendix,  we collect some mathematical tools needed along the way. In particular  we  prove a  relation between     single site contexts and contexts of finite sets of sites. 

\vskip0.5cm \noindent {\bf Acknowledgements.}  The authors thank
Antonio Galves for introducing them to the problem,    Lionel Moisan and the referees for helpful comments.  Enza Orlandi
thanks Antonio Galves and Roberto Fern\'andez for helpful discussions
at the starting of this project. She further acknowledges  enlightening discussions with Anton
Bovier,  Marzio Cassandro,  Stephan Luckhaus and Errico Presutti.   
  The authors have been partially supported by Prin07:
20078XYHYS (E.O.), ANR-08-BLAN-0220-01 (E.L.) and Galileo 2008-09
(E.L. and E.O.), GREFI-MEFI.
  
Ce travail a b\'en\'efici\'e d'une aide de l'Agence Nationale de la Recherche
portant la r\'ef\'erence ANR-08-BLAN-0220-01.

\section{Variable-neighborhood   random fields} \label{notation}

We consider the $d$ dimensional lattice $\Z^d$. The points $i \in
\Z^d$ are called sites, $\|i\|$ denotes the maximum norm of $i$,
i.e.~for $i = (i_1, \ldots , i_d ) , $ $\| i\| = \max ( |i_1| , \ldots
, | i_d | ) $ is the maximum of the absolute values of the coordinates
of $i$.  The cardinality of a finite set $ \Delta $ is denoted by $
|\Delta|.$ %and we denote by $d(\Delta)= \max \{ \|i-j\|, i \in \Delta, j \in
%\Delta \}$ the diameter of the set $ \Delta$.  
The notations $\subset$
and $\Subset $ denote inclusion and strict inclusion.  Subsets of
$\lat$ will be denoted with uppercase Greek letters.  If $\Lambda$ is a
finite set, we write $\Lambda\Subset\lat$.

A {\it random field} $X$ is a family of random variables 
indexed by the sites $i $ of the lattice, $\{X_i: i \in \Z^d\}$,  
where each $X_i$ is a random variable taking values in
a  finite set  $\cA$. 
 
We denote the  set of all possible 
configurations of the random field by
$$\Omega =\cA^{\Z^d},$$ 
where $\Omega$ is endowed with the product topology.
We adopt the following notational conventions.
%``$\card{\ }$'' stands for cardinality;
% regions will appear as subscripts,
We write $\omega_\Lambda\in   \cA^\Lambda$ 
for the restriction of the configuration $\omega $ to 
the subset $\Lambda .$ 
If $\Lambda = \{ i \} $ is a singleton, we shall write  
$ \omega(i) $ for $\omega_{\{ i\}}$.
Configurations defined by regions are factorized with
omitted
subscripts indicating completion to the rest of the lattice:
$\omega_\Lambda\eta_{\comp{\Lambda}}=\omega_\Lambda\eta$.  
We call \emph{local configurations} the elements of 
$\cup_{\Lambda\Subset\lat} \cA^\Lambda$.

We  identify  the random field  $ \{X_i: i \in \Z^d\}$ with  the coordinate maps 
$X_i$ by $X_i (\omega ) = \omega (i),$
for any $\omega \in \Omega ,$ and from now on we will use this {\it canonical version} of the
random field.  We  define  the  following $\sigma-$algebras:  For any $\Gamma \subset \Z^d, $ 
let 
$$ {\cal F}_\Gamma = \sigma \{ X_i : i \in \Gamma \} 
\mbox{ and }  {\cal F} = \sigma \{ X_i : i \in \Z^d \} .$$
%For any fixed $x \in \Z^d,$ we denote by $T_x$ the $x-$shift defined by 
%$$ (T_x \omega) (y) = \omega ( x + y ) .$$  
In this setup a random field is just a probability measure on the product space $ ( \Omega,  {\cal F})$. 
This measure is
defined by local specifications. To  define them,   we recall the
following well-known notions in statistical mechanics, see Georgii (1988), \cite{geo88}.
  
\noindent
\begin{definition}
   A {probability kernel} on $(\Omega,\FF)$ is a function 
  $\Gamma(\,\cdot\mid\cdot\,) : \FF\times \Omega \longrightarrow [0,1]$ such that
 \begin{itemize}
  \item[(a)] $\Gamma(\,\cdot\mid\omega)$ is a probability measure on $(\Omega,\FF)$,
 for each $\omega \in\Omega$,
   \item[(b)] $\Gamma(A\mid\cdot\,)$ is a $\FF$-measurable function for each $A\in\FF$.
   \end{itemize}
  \end{definition}
   
  \noindent
  \begin {definition}    A  {specification} on $(\Omega, \FF)$  is a family 
    $ \g= \{\g_{\Lambda}\}_{\Lambda \Subset  \Z^d}  $ of probability kernels on $(\Omega,\FF)$ 
    such that  
  \begin{itemize}
  \item[(a)] For each $\Lambda\Subset\lat$ and each $ A \in \FF$, the
    function $ \g_{\Lambda}(A\mid\cdot\,)$ is $
    \FF_{\comp\Lambda}-$measurable,

\item[(b)] For each $\Lambda\Subset\lat$ and each $ A \in
  \FF_{\comp\Lambda}$, $ \g_{\Lambda}(A\mid\omega)= 1_{A}(\omega)$,

\item[(c)] For any pair of regions $ \Lambda$ and $\Delta$, with
  $\Lambda \subset \Delta\Subset\lat$, and any measurable set $A$,
 
  \begin{equation}
 \int \g_{\Delta}(d\omega '\mid \omega) \; \g_\Lambda(A\mid\omega ')
 \;=\; \g_{\Delta}(A\mid\omega)
 \label{eq:3}
\end{equation}
for all $\omega\in\Omega$.
\end{itemize}
  \end{definition} 
 
  \noindent
  \begin{definition}
  A probability measure $\mu$ on $(\Omega,\FF)$ is consistent with
    a specification $\g$ if for each $\Lambda\Subset\lat $ and for each
    $A\in\FF ,$
    \begin{equation}
  \int \mu(d\omega) \; \g_\Lambda(A\mid\omega)
  \;=\; \mu(A) .
  \label{eq:3.1}
  \end{equation}
\end{definition}
We now define the   Variable-neighborhood   random fields. 
  \noindent
  \begin{definition}  \label {D23} {\bf Variable-neighborhood  random field} \\
    Let $\mu $ be a probability measure on $(\Omega,\FF)$ consistent
    with the specification $\g .$ Then $\mu $ is a variable-neighborhood random field if for any $\Lambda \Subset \Z^d $ and
    for $\mu -$almost all $\omega_{\comp\Lambda} $ the following
   holds:
   there exists  $ \G = \G(\omega) \Subset \Z^d$ 
 such that   
 $$  \gamma_\Lambda(\cdot \mid\omega_{\comp\Lambda}) =
 \gamma_\Lambda(\cdot \mid \omega_{\Gamma }),$$ 
and 
for all $\tilde \G \subset \Z^d ,$ if $\gamma_\Lambda(\cdot \mid\omega_{\comp\Lambda}) =
 \gamma_\Lambda(\cdot \mid \omega_{\tilde \Gamma }),$ then
 $     \G \subset \tilde \G.$
We  denote 
 \begin{equation} \label {rm1} \supt_\Lambda (\omega) = \G(\omega)  \qquad \hbox {and} \qquad c_\Lambda (\omega) = \omega_{\supt_\Lambda (\omega )},    \end {equation}
 the restriction of $ \om$ on
  the set $\supt_{\Lambda }(\omega)$. 
   \end{definition}
   
    \begin{rem}   In Definition \ref  {D23} the requirement  that  
  if   $\gamma_\Lambda(\cdot \mid\omega_{\comp\Lambda}) =
 \gamma_\Lambda(\cdot \mid \omega_{\tilde \Gamma }),$ then
 $     \G (\omega) \subset \tilde \G $   allows to identify in an unambiguous way the random set  $ \Gamma (\omega)$ on which  $ \gamma_\Lambda(\cdot \mid\omega_{\comp\Lambda}) $ depends on. \end {rem}
  \begin{rem}
  %Considering $\supt_{\Lambda } (\omega ) $ 
  We  call the  Variable-neighborhood   random fields    also  Parsimonious  Markov random fields.  Namely    $\gamma_\Lambda ( \cdot | \omega_{\comp\Lambda} ) $
  depends only on $ \omega_{ \supt_{\Lambda } (\omega )} $ and   we do not  need to inspect   the whole
  configuration $ \omega_{\comp\Lambda} $ in order to decide about the configuration of symbols within $\Lambda$. Indeed it is sufficient to inspect $ \omega_{ \supt_{\Lambda } (\omega )} .$
\end{rem}
  
   According to Definition \ref{D23} there might be a set of realizations
   of $\mu -$measure zero so that $| \supt_{\Lambda}(\omega)| =
   \infty$.  From now on we assume that {\it for all } $ \om \in
   \Omega$, $ \supt_{\Lambda}(\omega) $ is a finite set.  This means
   that for all $ \omega \in \Omega$, $ \gamma_\Lambda(\cdot
   \mid\omega_{\comp\Lambda})$ does only depend on a finite, but
   random neighborhood of $\Lambda$.  When for some $ \Gamma_0 \Subset
   \Z^d$ , $ \supt_{\Lambda}(\omega)= \Gamma_0$ for all $ \omega $,
   then $\mu $ (respectively, $X$) is a Markov field with basic
   neighborhood $\Gamma_0$. 
    Define the
   $\sigma-$algebra
   \begin{equation}\label{F1} \FF_ {\supt_{\Lambda} }= \left \{ A \in
       \FF: \forall \Gamma \subset \Z^d: \{\supt_{\Lambda} = \Gamma \}
       \cap A \in \FF_\Gamma \right \}. \end {equation}
   Then  for  all   $ \omega_\Lambda \in \cA^{\Lambda}$, 
   $   \gamma_\Lambda (\{\omega_\Lambda \} | \cdot)$ is a measurable function with respect 
   to  $ \FF_ {\supt_{\Lambda} }$.  
  
  In analogy to the terminology used for one dimensional variable length  Markov chains we can rephrase  the  Definition \ref{D23} using the concept of 
  family of contexts.   This   generalizes the notion of context trees    to more than one dimension. 
   \noindent
   \begin{definition} \label {D22}    {\bf The  family of contexts associated to the specification $\g$ }\\
  For $\Lambda \Subset \Z^d$ and $ \omega \in \Omega$ we denote  by  $ c^\g_{\Lambda} (\omega)=c_{\Lambda} (\omega) $,    see \eqref {rm1},  the $\Lambda-$context of $\omega$  associated to the specification $\g $. 
   We write $ \tau^{(\L)} \equiv  \tau^{(\L)}_{\g}= \{ c_{\Lambda}(\omega) , \omega \in \Omega
  \}$ for the { family of $\L$-contexts.} Under our assumptions,
  \begin{equation} \label {D25} \tau^{(\L)} \,\subset\,\bigcup_{\Gamma
      \Subset\lat\setminus \Lambda } \cA^\G .
  \end{equation} 
   \end{definition}
 We use the short-hand notation $ c_i ( \omega ) $ for $c_{\{i \} }
  (\omega ) $,  $\supt_i (\omega) $ for $\supt_{\{ i\}} (\omega )$
  and      $ \gamma_i (a | \omega ) $ for $\gamma_{\{i\}} ( \{a\} | \omega),$   $i \in \Z^d .$

   \noindent
   \begin{rem} It is immediate to verify from Definition \ref
     {D23} that the family $ \tau^{(\L)} $ has the following
     properties:

     \begin{itemize}
     \item %{\bf Suffix property}:
       No element of $\tau^{(\L)}$ is restriction of any other element
       of $\tau^{(\L)}$: If $\eta_\Gamma$ and
       $\widetilde\eta_{\widetilde\Gamma}$ both belong to
       $\tau^{(\Lambda)}$,  $\Gamma \subset \widetilde\Gamma$ and 
       $ \eta_\Gamma = \widetilde\eta_\Gamma,$ then
       $\Gamma = \widetilde\Gamma .$
  
     \item %{\bf Completeness}:
       $\tau^{(\L)}$ defines a partition of $\cA^{\lat\setminus \L }$,
       that is, for each $\omega\in\cA^{\lat\setminus \L }$ there
       exists a unique $\G\subset\lat\setminus \L $ such that
       $\omega_\G
       \in\tau^{(\L)}$. 
     \end{itemize} \end{rem}
 
   \noindent 
   In this way the family of local specifications associated to $\mu $
   is
   \begin{equation}\label{s2}\gamma= \{ \gamma_\Lambda (\cdot|
     c_{\Lambda} (\omega)), \Lambda \Subset \Z^d; c_{\Lambda} (\omega)
     \}\end {equation}
   which leads to a more parsimonious description than the original 
   \begin{equation}\label{s1}
     \{ \gamma_\Lambda (\cdot|  \omega_{\comp\Lambda})), \Lambda \Subset \Z^d;   \omega_{\comp\Lambda} \}. \end {equation}
  
 We close the section  presenting three examples.
  In  the first   one we embed a renewal process in a  one dimensional Variable-neighborhood      random field.   This example has been    suggested by    Ferrari and Wyner (2003), \cite {ferrariwyner}. 
    In the second example we construct  a two dimensional Variable-neighborhood  random field by specifying 
a   variable-neighborhood interaction potential.   In the third example, we define a two dimensional Voronoi cell interaction model. This example has been
inspired by Dereudre, Drouilhet and Georgii (2010), \cite{DDG}. 
\noindent
\begin{example} \label{ex:ex1} 
We consider  $\cA= \{0,1\}$.
Let   $ \{ X_n: n \in \Z \}$ be a stationary alternating renewal process taking values in $\cA ,$ i.e. the times  when   the process  switches  between $1$ and $0$ or $0$ and $1$ are independent and identically distributed random variables.  
They have the same distribution as the   random variable $T$ defined through 
$$ \P [ T = j] = c_1 \rho_1^j +  c_2 \rho_2^j,  \quad 0 < \rho_2 < \rho_1<1, \quad j \in \N. $$ 
Set  $ \mu = \E [T].$ 
This process  can be   realized   as  a  one dimensional Variable-neighborhood    random field.  To this aim define  for $a $ and $b$ in $\Z$  
\begin{equation}\label{v1}  R_b (\omega) = \inf \{ n > b+1 : \omega (n) \neq \omega (b+1) \}, \quad 
L_a( \omega ) = \sup \{ n < a-1 : \omega (n) \neq \omega (a-1) \}, \end {equation}
where $ \omega \in \cA^{\Z}$.
 The family    of  local specifications $ \gamma_{\{  [a, b] \}} $  indexed by $[a, b] \subset \Z $,  % $a\in \Z$, $b \in \Z$  is
%associated to  $\{ X_n: n \in \Z \} $
is  given as follows: 
 $$  \gamma_{\{  [a, b] \}} ( \cdot  \mid c_{ [a, b] } (\omega) ) =  \gamma_{\{  [a, b] \}} ( \cdot  \mid L_a (\omega) = - k , R_b(\omega) = l  ). $$
The context  $ c_{ [a, b] } (\omega)$  depends only on the neighbor sites   of $[a, b]$ which are all of the same type $0$ or $1 .$
%We can    compute   $  \gamma_{\{  [a, b] \}} ( \cdot  \mid c_{ [a, b] } (\omega) ) $.
  By standard calculus we obtain the following formulas for the one point specification  $  \gamma_{\{0\}} ( \cdot  \mid c_{0 } (\omega) ) .$ 
Write $\supt_0 (\omega ) = [ L_0( \omega ) , R_0 ( \omega ) ] \setminus \{ 0 \} .$ 
Then  
\begin{multline}  \label{final1}\gamma_{\{ 0\}} ( 1 \mid L_0( \omega ) = - k , R_0( \omega) = l, \omega (1) = \omega (-1) = 1 ) 
\\
= \frac{\left( c_1 \varrho_1^{l+k-1} + c_2 \varrho_2^{l+k-1} \right)  }{\left( c_1 \varrho_1^{l+k-1} + c_2 \varrho_2^{l+k-1} \right) +\left (  c_1  \rho_1^{k-1} +  c_2    \rho_2^{k-1} \right )\left(  c_1  \rho_1^{l-1} +  c_2    \rho_2^{l-1} \right )
\frac{  c_1 \varrho_1 + c_2 \varrho_2    }{\left(  \frac{c_1}{1- \varrho_1}\varrho_1 + \frac{c_2}{1- \varrho_2}\varrho_2  \right)^2    } 
  }
\end{multline}
and   
\begin{multline}  \label{final2}\gamma_{\{ 0\}} ( 1 \mid L_0( \omega ) = - k , R_0( \omega) = l, \omega (1) =0,  \omega (-1) = 1 ) 
\\= \frac{ \left(c_1 \varrho_1^k + c_2 \varrho_2^k  \right) \left(c_1 \varrho_1^{l-1} + c_2 \varrho_2^{l-1}  \right)  }{\left(c_1 \varrho_1^k + c_2 \varrho_2^k  \right) \left(c_1 \varrho_1^{l-1} + c_2 \varrho_2^{l-1}  \right) + \left(c_1 \varrho_1^{k-1} + c_2 \varrho_2^{k-1}  \right) \left(c_1 \varrho_1^{l} + c_2 \varrho_2^{l}  \right) } . 
\end{multline}
Due to the symmetry between $0$ and $1$, it is clear that with formulas \eqref{final1} and \eqref{final2}, we have  completely described  the 
one-point specification. 
 
In this example  the context  $c_{0 } (\cdot)$ is $\P-$almost surely finite, i.e.~there exists a subset of configurations of $ \P -$measure zero for which $| c_{0 } (\omega)| = \infty$. 
 \end{example}
 
We now give an example of a Variable-neighborhood random   field in dimension $d=2.$   In analogy with the   procedure  used in  statistical mechanics  we define a variable-neighborhood  specification by introducing  a   variable-neighborhood interaction potential.  \vskip0.5cm 
\noindent

\begin{example} \label{ex:ex2}
We consider $ \cA= \{-1,1\}$ and $d =2.$ In order to define the support  of  the    variable-neighborhood interaction potential
it is convenient to embed $\Z^2$ into $\R^2$. We partition $\R^2$ into  cubes   
of edge 1 centered  at $\Z^2$.  We say that two cubes are connected if they have one face in common. 
We denote by  $ \cR$  the set of all connected unions of such cubes, by $R $  an element of $\cR$ and by
$ \partial R $ the topological surface of  $R$. 
 
\noindent
We say that  $\G \subset \Z^2 $  is a polygon  if there exists $R \in \cR$ so that 
 $ \G= R \cap \Z^2$.   We denote by  $\partial \G = \{ i \in \G: d(i,  \partial R)\le \frac 12 \},$ 
where $ d(i, \partial R) = \inf \{ \| i - j \| : j \in \partial R \}$ and
$\| \cdot \| $ is the maximum norm introduced at the beginning of this section. Finally,
let $\hat \G $ be the interior of $ \G$,  $  \hat\G = \G \setminus  \partial \G$.    

We  say that $ \G$ is  a simple polygon if  $\partial \G $ is a  path in $\Z^2$ which does not
  cross itself  and  $ \hat \G \neq  \emptyset$.  Note that   $\partial \G $  can be   the union of disjoint  connected paths.  
  
\noindent
%To define a  variable-neighborhood interaction potential $K$  we  start defining its support. 
Given $ \omega \in \cA^{\Z^2} $ we define for each $i \in \Z^2$
$$   \Gamma^1_{i} (\omega)=    \cap \{ \Gamma \subset \Z^2 ,  \G  \hbox { simple polygon },   i  \in     \hat \Gamma,    \omega_ {\partial \Gamma} = 1 \; 
   \} . $$
In the above definition we do not require $ \Gamma$ to be finite. $ \Gamma$ could be equal to  $\Z^2$ in which case ${\partial \Gamma} = \emptyset$. 
In order to get a   bounded   interaction range, we  set 
$$\Gamma_i (\omega) = V_i (L ) \cap \Gamma_i^1 (\omega )\mbox{  and  } c^K_i (\omega) = \{ \omega _j: j \in \Gamma_i (\omega), \quad  j\neq i  \} ,$$
where
 $ V_i (L) = \{ j \in \Z^2 : \| i - j \| \le L \} .$ In the above definition, the superscript $K$ underlines the fact that
 the context $c^K_i (\omega)$ for the interaction  might not be  the same  as the context $c_i (\omega) $ associated to the specification.  
  Let    $ \{ J_n , n \in \N \} $ be a collection of real numbers
 and  $ | \Gamma_i
 (\omega)| $   the cardinality  of $ \Gamma_i (\omega) .$  We define 
 the  variable neighborhood interaction  $ \{ K^i (\omega),  i \in \Z^2\} $  as following:
   $$ K^i (\omega ) = K^i ( c^K_i(\omega))=   J_ {|\Gamma_i (\omega) | } \prod_{j  \in \Gamma_i (\omega) }  \omega(j).   $$
   By construction,
 the interaction is summable:  
\begin{equation}
\sup_{i \in  \Z^2 }  \sum_{j \in \Z^2  } \;
\sup_{\omega \,:\, \Gamma_j (\omega)\ni i}  |K^j ( \omega ) | \;<\; \infty\;. 
\label{eq:1}
\end{equation}
Denote 
$$ H_{\Lambda}  (\omega_{\Lambda}, \omega_{\Lambda^c})=  
- \sum_{ \{ i \in \Z^2 : \Lambda \cap \Gamma_i (\omega) \neq \emptyset
  \} } K^i (\omega). $$
  The Variable-neighborhood   random field  $\mu$ is determined by the 
   following    family of local specifications
 \begin{equation} \label {mars3} \g_\Lambda(\{\omega_{\Lambda}\} \mid\omega_{\Lambda^c}\,)=  
 \frac 1{Z^{\omega_{\Lambda^c}}} \exp\{-\beta H_{\L}(\omega_{\Lambda}, \omega_{\Lambda^c} )\} , \quad \quad \omega_{\Lambda} \in \cA^{\L},  \omega_{\Lambda^c}\in \cA^{\L_c}, \end{equation}  where
$$ Z^{\omega_{\Lambda^c}} =
\sum_{\om_\L \in \{-1,1\}^\Lambda} \exp\{-\beta H_{\L}(\omega_{\Lambda}, \omega_{\Lambda^c} )\}
 .  $$ 
The  family of contexts  $c_\L(\omega)= \omega_{sp_\L(\omega)}$ associated to  $ \{\gamma_\Lambda \}_{\L}$,   
 defined in  \eqref {mars3},   is   %, due to \eqref{eq:freiburg1}, 
   determined by  $c_i(\omega) $ for  $ i \in \Lambda $, therefore by the knowledge of $ \omega$ only on   $ \supt_i (\omega )$. By  Definition    \ref {D23} and by (\ref{mars3}) we have that
\begin{equation}\label{eq:oho2}
 \supt_{i} (\omega ) = \left[ \bigcup_{j \in \Z^2} \{ \Gamma_j (\omega ) : i \in \Gamma_j (\omega) \}
  \cup \bigcup_{j \in \Z^2} \{ \Gamma_j (\omega^i ) : i \in \Gamma_j (\omega^i) \}
\right] \setminus \{i\} ,
\end{equation} 
where $\omega^i ( j) = \omega (j) $ for
all $ j \neq i, $ $ \omega^i (i) = - \omega (i) .$  This formula  gives the relation between the support of the context of the specification and the support of the
interaction.  We show in the appendix 
that the following identity holds:
\begin{equation}\label{eq:oho}
 \supt_{i} (\omega ) =  \left ( \Gamma^1_i (\omega) \cap V_i(2L)\right ) \setminus \{i\}, \qquad    c_i(\omega)= \om_{\supt_{i} (\omega )} . 
\end{equation} 
 Note that the context  associated to the family of the constructed specification   $ c_i(\omega)$ is different from  $ c_i^K(\omega),$ the context associated to the interaction.  
 \end {example}
  
 We now give a second example of a Variable-neighborhood random field in dimension $d=2$ in which contexts are no more bounded. This example is inspired by a recent  paper of Dereudre, Drouilhet and Georgii (2010), \cite{DDG} in which    Gibbs point processes on $ \R^d $ with geometry dependent interactions  are considered.  Such processes can be  realized as  Variable-neighborhood random fields.  Compared to their work, our setup is  simpler  
since we do not work on $\R^2$ but  on the grid $\Z^2.$    \vskip0.5cm 
\noindent

\begin{example} \label{ex:ex3}
We consider $ \cA= \{0,1\}$ and $d =2.$ 
  For  $\om \in  \cA^{\Z^2}$ we denote by $ C(\om) = \{ i \in \Z^2 : \omega (i) = 1 \}.$    We embed $\Z^2 $ into $\R^2 .$  For any fixed $\omega \in \cA^{\Z^2} $ with $C(\om) \neq \emptyset $ and $ i \in  C(\om), $ let 
$$ Vor_i(\omega)  := \{ y \in \R^2 : \| i - y \|_2 \le \| k - y \|_2,\; \forall k \in  C(\om), \, k \neq i  \} $$ 
be the Voronoi cell associated to a point $ i \in  C(\om).$  
If   there exists no such  $ k \in   C(\om) $, then set $Vor_i(\omega)= \Z^2 \setminus \{i \}$.  We denote by   $\| \cdot \|_2 $ the  $L^2-$norm on $\R^2.$  
\noindent
For a bounded function $ \Phi : {\cal P } ( \R^2 )  \to \R $ defined on all subsets of $\R^2 ,$ we define for any $i \in C(\om),$
$$ K^i (\omega ) =  \Phi ( Vor_i (\omega )) . $$
For $i \notin  C(\om),$ we set  $ K^i (\omega )= 0 .$ When $C(\om) = \emptyset $ set    $ K^i (\omega )= 0 $ for all $i \in \Z^2$.  
Finally, let 
$$ H_{\Lambda}  (\omega_{\Lambda}, \omega_{\Lambda^c})=  
- \sum_{ \{ i \in \Z^2 : \, i \in  C(\om), \, \Lambda \cap Vor_i (\omega) \neq \emptyset
  \} } K^i (\omega). $$
Clearly, the range of interaction is not bounded in this case. We have to ensure that the interaction is summable. Since for any fixed $i \in \Z^2 $   and $k \in \N $ 
$$ card \{ j \in \Z^2 : i \in \Z^2 \cap Vor_j (\omega ) , \| i - j \|  = k \} \le  8 k,$$ 
 where    $ \| \cdot \|$ is the $L^\infty-$norm, see  at the beginning of Section 2,  it suffices to impose that for any  $ A \subset \R^2, $     $ \Phi ( A ) \le  \frac   C { diam (A)^{2+ \e}},$   for some constant $C$ and   $ \e>0.$  Then,    
\begin{equation} \label {eq:1aa}
\sup_{i \in  \Z^2 } \sup_\omega \sum_{j \in \Z^2 : j \in C( \omega ) , \, i \in Vor_j (\omega) } \;
  |K^j ( \omega ) | \;<\; \infty\;. 
\end{equation}
% therefore  the interaction is summable. 
  The Variable-neighborhood   random field  $\mu$ is   then  determined by the 
   following    family of local specifications
 \begin{equation}  \g_\Lambda(\{\omega_{\Lambda}\} \mid\omega_{\Lambda^c}\,)=  
 \frac 1{Z^{\omega_{\Lambda^c}}} \exp\{-\beta H_{\L}(\omega_{\Lambda}, \omega_{\Lambda^c} )\} , \quad \quad \omega_{\Lambda} \in \cA^{\L},  \omega_{\Lambda^c}\in \cA^{\L_c}, \end{equation}  where
$$ Z^{\omega_{\Lambda^c}} =
\sum_{\om_\L \in \{0,1\}^\Lambda} \exp\{-\beta H_{\L}(\omega_{\Lambda}, \omega_{\Lambda^c} )\}
 .  $$ 

 \end {example}
  One can   generalize in a relatively straightforward way  the last  two examples to $d\ge 3$.

\vskip0.5cm 
 We  are  interested in estimating the  support of the context $\supt_\L(\omega)$ for a given set of sites $\Lambda
\Subset \Z^d $ and a given observation $\omega .$  
Proposition \ref {p1} of the Appendix shows     that
    $ \gamma_\Lambda(\,\cdot\mid c_\Lambda (\omega )\,) $ can be derived from the   one 
   point specification $ \gamma_i(\,\cdot\mid c_i (\omega )\,) $ and 
  that for $ \Lambda \Subset \Z^d$, we have 
\begin{equation}\label{eq:freiburg1}
  \supt_{\Lambda} (\omega)=  \cup_{\omega_\Lambda} \left ( \cup_{i \in \Lambda} \supt_{\{i \}}(\omega)\right) \setminus \Lambda.   
 \end{equation}
Hence, in order to estimate $\supt_\L (\omega ),$ it is sufficient to estimate the context for single sites, i.e.~$\supt_i (\omega).$ 
To implement the  estimation procedure   we need  translation covariant models.  
 \vskip 0.5cm \noindent For any fixed $i \in \Z^d,$ we denote by   $\tau_i: \Z^d \to  \Z^d$ the $i-$shift 
defined by $\tau_i(j)= i+j$. This naturally induces on $\Omega$ the $i-$shift 
$T_i: \Omega \to \Omega$ defined  by 
$$ (T_i \omega) (j) = \omega (\tau_i (j) ) = \omega (i+j)  \quad \forall j \in \Z^d.$$ 

\begin {definition}  \label {mars2}  
  A   Variable-neighborhood   random field   $\mu$ on $(\Omega,\FF)$, determined by a family of local specifications  $\{\gamma_{  \Lambda}\}_\L , $   
is   translation covariant   if  for all $\L \Subset \Z^d$ and 
    for  all $ \omega \in \Omega$ 
$$ \gamma_{  \Lambda} (\cdot |\omega)=  \gamma_{\tau_i\Lambda} (\cdot |T_{i}  \omega),\quad i \in \Z^d $$
where $\tau_i \Lambda= \Lambda +i$.  
\end {definition} 
  In the following we will consider only   translation covariant Variable-neighborhood   random fields.  This implies that $
\gamma_{i} (\cdot |c_i(T_i (\omega ) ))= \gamma_{0 } (\cdot |c_{0} (
\omega)) $. 
 
\section{Main Results and Estimation procedure}\label{section3}
\noindent
In Section 2 we introduced the notion of Variable-neighborhood   random
fields. Such a random field is completely determined by the one point specification. It
would therefore be interesting to estimate $\supt_{i }(\omega)$, i.e. the set of points in $ \Z^d$  which enables to determine the value of the   symbol  at the site $i$. 
This requires, however, to estimate too many unknown parameters. Therefore we
are less ambitious and estimate the radius of the smallest ball containing $\supt_{i}(\omega)$.  
For $\ell \geq 1$
and $i \in \Z^d,$ define
\begin{equation} \label {set1}  V_i (\ell ) = \{ j \in
  \Z^d : || i-j || \le \ell \}   \quad \hbox { and} \quad  V^0_i (\ell ) = V_i (\ell ) \setminus \{i\}   .  \end {equation}
We also write
$$ \partial V_i ( \ell ) = \{ j \in \Z^d : \| i - j \| = \ell \} .$$
Then we define the length of the context of site $i$ by  
\begin{equation}\label{e1}   l_i (\omega) = \inf  \{  \ell>0:     \supt_{i}(\omega) \subset V_i(\ell)  \} .   \end {equation}
Note that $ l_i (\omega)$ is a stopping time with respect to the filtration  $ (\cG^i_n)_n = (\FF_{ V_i(n) })_n$.

Recall that $ \omega \in \Omega = \cA^{\Z^d} $ stands for a generic configuration of the
field.  In order to distinguish  between    generic configurations and observed data,  we  
will   denote  the   observed data by $\sigma.$ Our statistical inference is based on observations of    the 
Variable-neighborhood   random field 
$\mu  $ over an increasing  and absorbing sequence of finite regions $ \Lambda_n \subset \Z^d ,$ i.e.~$ \Lambda_n \subset \Lambda_{n+1} \subset \Z^d $ for
all $n ,$ and for all $ \Lambda' \subset  \Z^d$, there exists $n$ such that $  \Lambda' \subset \L_n$. 
 
Hence, at step $n,$ the sample is $\sigma_{\Lambda_n},$ where $\sigma_{\Lambda_n}$ is the fixed realization of $\mu $
in restriction to $\Lambda_n.$ We will construct our
estimators based on sites within some security region $\ll  \subset \Lambda_n ,$  where 
\begin{equation}\label{e4}
  \ll = \{ i \in \Lambda_n : V_i ({k(n)})  \subset \Lambda_n \} \end{equation}
with 
\begin{equation}\label{e6} k(n) =  \left ( \log |\Lambda_n| \right )^{\frac1 {2d}}. \end{equation}

In order to estimate $l_i (\omega ),$ we have to compare 
the neighborhood configuration of site $i$ with   the neighborhood configurations   of different sites $j$ for all  $j \in \ll $.    To  do so  we  %introduce the
%following local coordinates.
define for any fixed $i \in \ll $ and any $ 1 \le \ell \le k(n),$   
\begin{equation}\label{eq:wk}
  X_i^\ell (\omega)= \{  X_i^\ell (\omega) (j)  = \omega ( i + j),  \quad j : 0 < || j || \le \ell\}, 
\end{equation} 
hence  $ X_i^\ell (\omega)$ is the configuration around $i$ in a box of edge $\ell$. 
In terms of the shift operator,   $X_i^\ell (\omega) = (\omega \circ \tau_i ))_{V_0^0 (\ell)} $, i.e. this is the restriction of $ T_i \omega $ to $V_0^0 (\ell ) .$ 
  We stress  that $X_i^\ell $ does not depend on $\omega (i),$ the center of the observation window, and this   is important  to  our purposes.
We shall use the short-hand notation  
$$ X_i^\ell (\omega) = \omega_i^\ell.  $$
%Hence, given the observation $\sigma ,$ we observe, for any fixed %site $i \in \ll ,$ $\sigma_i^\ell = X_i^\ell (\sigma), $ which is the local %pattern of the field around $i$ given  the  observation $\sigma .$ 
For any $ 1  \le \ell \le k(n),$ for any fixed configuration $\eta \in \cA^{V^0_0 ( \ell ) }, $ let
\begin{equation}\label{Do.1}  N_n (\eta)   = 
 \sum_{ j \in \ll }  1_{\{ X_j^{\ell} = \eta  \} }   \end{equation}
be the total number of occurrences of $\eta $ within $\ll .$  Moreover, for any fixed value $a \in \cA ,$ we write 
\begin{equation}\label{Do.2}   N_n (\eta , a) =  \sum_{ j \in \ll }  1_{\{ X_j^{\ell} = \eta , X_j = a \} } .  \end{equation}
In particular for the observed data $ \sigma$, $ \sigma_i^\ell$ is the data observed around the site $i$ in a ball of radius $ \ell$,  $N_n ( \sigma_i^\ell )  $ 
is the total number of occurrences of the local pattern around $i$ within $\ll .$ By construction $ N_n (\sigma_i^\ell)  \ge 1$.
Note that  $ N_n (\sigma_i^\ell , a  )  $ could  be zero. 

Let $\gamma: \cA \times \cA^{V^0_0(\ell)} \to [0,1]  .$  $\gamma$ is interpreted as possible one-point specification of a hypothetical Markov random field for which the corresponding context is
contained in $ V_i (\ell)$.
For any site $i$, under the hypothesis that its context  is contained in    $ V_i (\ell),$ 
we define the {\it  pseudo-likelihood of $\gamma$} as follows:
 \begin{equation} \label {ps1}
 PL^{(i, \ell)}_n (\gamma)= \prod_{j \in  \ll, \;  X_j^{\ell} = \sigma _i^{\ell}  }  \gamma (X_j | X_j^{\ell})  = 
  \prod_{ a \in \cA}  \gamma ( a | \sigma_i^{\ell})^{ N_n ( \sigma_i^{\ell}, a)} ,
\end{equation}
 where we restrict the product to all sites $ j \in \ll $ in order to be sure that $ V_j (\ell ) $ is still contained inside the observation window $\Lambda_n .$ 
Maximizing \eqref {ps1}  with respect to $\gamma$ under the  constraint 
  $$ \sum_{ a \in \cA}  \gamma ( a | \sigma_i^{\ell}) =1 $$
 gives  the following  estimator of the one-point specification 
  \begin{equation}\label{eq:estimator}
    \hat p_n (a | \sigma_i^{\ell}) = \frac{N_n (\sigma_i^{\ell}, a)}{ N_n (\sigma_i^{\ell})}.  
  \end{equation}
Analogously, we can define for any fixed configuration $\eta \in \cA^{V_0^0 (\ell ) }, $
\begin{equation}\label{eq:estimatorbis} \hat p_n ( a|\eta ) =  \left \{ \begin {array}{ll} \frac{N_n (\eta, a)}{ N_n (\eta)} & \mbox{ if } N_n (\eta) > 0, \\
0 & \mbox{ otherwise.} \end {array} \right. 
\end{equation}

 \noindent
 \begin{rem} 
 Not all $\gamma$ satisfying $ \sum_{ a \in \cA} \gamma ( a
 | \sigma_i^{\ell}) =1 $ are possible one-point specifications; one
   point specifications have to satisfy additional conditions,  which  are collected in the appendix, see 
       \eqref {eq:cons},  and which  are not considered here.  However, we define the
      pseudo-likelihood also for $\gamma$ not satisfying these additional
   conditions.
  \end{rem}
 \medskip
% 
%  It is easy to verify that $\hat p_n (a | \sigma_x^{\ell})$ defined
%  in \eqref {eq:estimator} is the only point of maximum of $ PL^{(x,
%    \ell)}_n (\gamma)$ defined in \eqref {ps1}. Namely, maximizing $
%  PL^{(x, \ell)}_n (\gamma)$ is equivalent to maximizing $ F(\gamma):=
%  \log \frac { PL^{(x, \ell)}_n (\gamma) } { PL^{(x, \ell)}_n ( \hat
%    p_n (a | \sigma_x^{\ell}) ) }$.  We have that $F( \hat p_n (a |
%  \sigma_x^{\ell}) )=0$. So it will be enough to show that $ F(\gamma)
%  <0$ for all $ \gamma \neq \hat p_n $ and $ \sum_{ a \in \cA} \gamma
%  ( a | \sigma_x^{\ell}) =1 $. This is easily proved.  We have that,
%  avoiding to write explicitly the dependence on $\sigma_x^{\ell}$ and
%  applying the inequality $ \log x \le 2 (\sqrt x -1) $ when $x \ge 0,
%  $
%  \begin{equation} \begin {split} & F(\gamma)= \sum_a N_n (a) \log
%      \left ( \gamma (a) \frac {N_n} { N_n (a)} \right ) \le 2 \sum_a
%      N_n (a) \left ( \sqrt {\gamma (a) \frac {N_n} { N_n (a)} } -1
%      \right ) \cr &   %2 \sum_a N_n (a) \left ( \sqrt {\gamma (a) N_n
%          %N_n (a)} -1 \right )
%           = - \sum_a \left ( \sqrt {\gamma (a)
%          N_n }- \sqrt { N_n (a) } \right )^2.
% \end {split}   \end {equation}  
%\end{rem}
%\medskip

Thus, given the sample  $ \sigma_{\L_n}$, the logarithm of the maximum  pseudo-likelihood   of $\g$ is   the following  quantity:
 \begin{equation}\label{eq:loglikelihood1}
\log  MPL_n (i,\ell)=     \sum_{ a \in \cA } N_n (\sigma_i^{\ell}, a) \log  \hat p_n (a| \sigma_i^{\ell}) . 
\end{equation}
The decision if for a given $i$ the context has radius $ \ell-1$
rather than $\ell $ is based on the Kullback-Leibler information.
% $D(.,.) $ : 
We introduce
 \begin{equation}\label{eq:loglikelihood}
\log L_n (i,\ell)  =    \sum_{ v \in \cA^{\partial V_0(\ell)   } } N_n (\sigma_i^{\ell-1}v) D(\hat p_n (\cdot | \sigma_i^{\ell-1} v ), \hat p_n(\cdot | \sigma_i^{{\ell}-1} )),
\end{equation}
where we sum over all possibilities of extending $\sigma_i^{\ell -1} $ 
to a configuration $\sigma_i^{\ell-1}v$ of radius $\ell $ and where
$$ D(\hat p_n (\cdot | \sigma_i^{\ell-1} v ), \hat p_n(\cdot | \sigma_i^{{\ell}-1} ))=
 \sum_{ a \in \cA } \hat p_n (a| \sigma_i^{\ell-1} v ) \log \left[  \frac{\hat p_n (a| \sigma_i^{\ell-1} v )}{\hat p_n(a| \sigma_i^{{\ell}-1} ) }\right] 
$$
is the Kullback-Leibler information.
Note that $ \log L_n (i,\ell)$ is a function of $\sigma_i^{\ell - 1 },$ but not of $\sigma_i^\ell .$ 
We rewrite it as follows:
\begin{equation}\label{eq:loglikelihood3}
\log L_n (i,\ell)=   \sum_{j: \; X_j^{\ell-1} = \sigma_i^{{\ell}-1}  } \; \frac{1}{N_n (X_j^\ell) }\sum_{ a \in \cA } N_n (X_j^{\ell}, a) \log \left[  \frac{\hat p_n (a| X_j^{\ell})}{\hat p_n(a| X_j^{{\ell}-1} ) }\right] . 
\end{equation}
Finally note that  
\begin{equation}\label{eq:loglikelihood4}
\log L_n (i,\ell)=   \left [  \sum_{j : \;  X_j^{\ell-1} = \sigma_i^{{\ell}-1}  } \frac{1}{N_n (X_j^\ell) }  \log  MPL_n (j,\ell) \right ]- \log  MPL_n (i,\ell-1). 
\end{equation}
 Now we start from $ \ell = R_n$ and proceed successively from $\ell $ to $\ell - 1.$ The log likelihood ratio statistics $\log L_n (i, \ell ) $ will be basically equal to zero for all $ \ell > l_i (\sigma ) .$ The first range at which $\log L_n (i, \ell ) $ is significantly different from zero is a range such that 
$ \hat p_n (\cdot | \sigma_i^{\ell - 1 }) \neq  \hat p_n (\cdot | \sigma_i^\ell ) $ in which case it is
reasonable to suppose that $\ell =  l_i (\sigma ) .$  

Before formalising this intuition in the definition of the estimator, for technical reasons we have to introduce the following security diameter
\begin{equation}\label{eq:rn}
R_n=  \left [ \left (\log |\ll | \right )^{\frac 1 {2d}} \right ], 
\end{equation}
where $[\cdot]$ denotes the integer part of a number. Note however that 
$$ R_n / k(n) \to 1 \mbox{ as } n \to \infty ,$$
where $k(n)$ was defined in \eqref{e6}. 
We are now able to define the estimator of the context length function.
\noindent
\begin {definition}  {\bf The estimator} 
Given the observation $\sigma_{\Lambda_n} ,$  for any $i \in \ll$, see \eqref {e4},  the  estimator of
$l_i(\sigma )$, defined in   \eqref {e1},    is the following random variable
 \begin{equation}\label{eq:estimator2}
\hat l_n (i) =\hat l_n (i, \sigma ) = \min \{ \ell = 1, \ldots, R_n - 1 : \forall \, k > \ell , \; \log L_n (i, k ) \le pen (k , n)   \},
\end{equation}
whenever $\{ \ell = 1, \ldots, R_n - 1 : \forall \, k > \ell , \; \log L_n (i, k ) \le pen (k , n)   \} \neq \emptyset $.   Otherwise we set  $\hat l_n (i)= R_n$.  In the above definition, 
 \begin{equation}\label{eq:pen1}  pen (\ell , n  ) = \kappa |\cA|   |\cA|^{| \partial V_0(\ell) |   }  \log | \Lambda_n| , \end{equation} 
and $\kappa $ is a positive constant that can be chosen freely,
provided it is at least of the order given in \eqref{eq:kappa}.
\end {definition} 
In other words, the above estimator chooses the minimal length $\ell $ such that all sites which are
relevant to determine  the value of the  symbol  at site $i $ belong to a ball of radius $\ell .$  
Once we have estimated the context length function, the underlying context $c_i (\sigma )$ is then estimated by  
$$\hat c_{n,i} (\sigma) = \sigma_{ V_i^0 (\hat l_n (i))} ,$$
and the corresponding one point specification by $ \hat \gamma_{n,i} ( a | \sigma ) = \hat p_n (a | \hat c_{n,i} (\sigma) ) .$

\begin{rem}
\begin{enumerate}
\item In one dimension the above penalization  term  is 
independent of  $\ell, $ since in this case $ |\cA|^{| \partial V_0(\ell) |   }  = |\cA|^2 .$ This leads to a penalty term 
$$ pen (n) = \kappa | \cA|^3 \log | \Lambda_n| .$$
\item
 Once the statistician has determined the radius of the context $\ell = l_i (\sigma) $ by means of the estimator $\hat l_n (i), $ is
is possible, in a second step, to determine the geometry of the context, i.e. to estimate $c_i ( \sigma) $ itself. This can be done 
by adapting the estimator of Csisz\`ar and Talata (2006), \cite{CT}, to our setup where the penalty can be restricted to all shapes 
contained in $V_i ( \ell ) $ for $ \ell = \hat l_n (i) $. 
\end{enumerate}
\end{rem}

\subsection{Main results}
The estimator  $\hat l_n (i)$ depends on the penalization term,  \eqref {eq:pen1},  therefore on  the choice of the constant $\kappa$. 
  Choose $ \delta > 2^d \log
| \cA | \frac{3 e}{4 q_{min} } $ and   define  
\begin{equation}\label{eq:kappa}
  \kappa= \kappa(\delta)= 5^{d} \left(\frac{3}{2}\right)^{1/2} \delta .
\end{equation}
 For the estimator defined in this way the following theorems are our  main
 results.
 
\noindent
\begin{ass}\label{ass:ergod}
The local specification is positive. We define
\begin{equation}\label{e3}  q_{min} = \inf_{a \in \cA} \inf_{ \omega \in \Omega} \gamma_0 ( a  |  \omega) >0. \end{equation}
\end{ass}

\noindent
\begin{theo}\label{theomain1}  {\bf (Overestimation)}  Let $ \mu$ be a translation covariant Variable-neighborhood   random field  
for which  Assumption \ref{ass:ergod} holds.  For any 
$\e
  > 0 $   there
  exist $n_0= n_0(\e , \delta, q_{min} ) $ and $ c(\delta)= c(\delta , q_{min}
  ),$ so that for any $n \geq n_0$ the probability of overestimation is bounded by
  \begin{eqnarray} \label{eq:badupperbound} && \mu \left[ \exists i
      \in \ll : \hat l_n (i) > l_i (\sigma )\right] \le \nonumber
   \\
   && \quad  C(d)  (\log  | \ll | )^{\frac{d+1}{2d }}   \cdot  \exp \left(  -c(\delta)  \sqrt{   \log | \ll |  }  \right)  % \nonumber \\
   % && \quad \quad \quad 
   + C(d) \exp \left( -  |\ll | ^{ 1 -
        \e } \right) ,
\end{eqnarray}
where $C(d)$ is a positive constant depending only on the dimension and where $ \ll$ is given in \eqref {e4}.  
\end{theo}
\noindent
\begin{rem}
 To obtain an upper bound in
\eqref{eq:badupperbound}  summable in $n$, we need a fast increase of the sampling
regions of order for example
$$ \log |\ll | \sim ( \log n^{1 + \varepsilon}  )^2 ,$$
which requires faster increase than choosing $\Lambda_n = [-
\frac{n}{2}, \frac{n}{2} ]^d .$
\end {rem}
 For bounding  the  probability of underestimation we need an additional assumption.   To this aim  define
$$ r(i,j)= \sup_{\omega , \omega':  \omega_{ \{j\}^c} = \omega'_{ \{j\}^c}}  \frac 12 \|\gamma_{i} ( \cdot| \omega )- \gamma_{i} ( \cdot| \omega ')\|_{\ TV} $$
where $\|\cdot\|_{\ TV}  $ denotes the total variation norm. By translation covariance  $ r(i,j)= r(0,i-j)$.   We    denote

\begin{equation}\label{eq:betal}                             
\beta (\ell) =  \sum_{k \in \Z^d : \|i\| > \ell}   r(0,k). 
\end{equation}
%We assume that the interactions are of finite but unknown range $L.$ 
 \noindent
\begin{ass}\label{ass:mixing}
We assume that  there exists  $L>0$ such that 
\begin{equation}\label{eq:fin}     r(0,i)=0  \quad \mbox{ for all }  \|i\| \ge L 
\end{equation} 
and
\begin{equation}\label{eq:erre}       r=  \sum_{i \in \Z^d \setminus \{0\}} r(0,i) <1. \end{equation}
\end{ass}

\noindent
\begin{rem}
  Condition (\ref{eq:fin}) implies that the observed
  random field is actually a Markov random field of order $L.$ The
  order $L,$ however, is unknown. We do not propose to estimate this
  unknown order $L.$ When passing to the parsimonious description
  (\ref{s2}), what we actually propose is to estimate, for every site
  $i,$ given the observation $\sigma ,$ the minimal order $l_i (\sigma)$ that we need
  in order to determine the specification at that site, given $\sigma
  .$ This is also called {\it Minimum Description Length} in the
  literature. However, if $l_i (\sigma) $ does not depend on the
  configuration, then our estimator naturally provides an estimator of
  $L.$
  
  Condition \eqref {eq:erre} is the Dobrushin condition which implies uniqueness of the measure $ \mu$, see Dobrushin (1968), \cite {D1}, \cite {D2}.
  
   For the Example \ref{ex:ex2}  and  Example \ref{ex:ex3},    thanks to the summability assumptions  \eqref{eq:1} and \eqref {eq:1aa}, there exists for both of them  a critical value of the temperature $\beta_c$ such that \eqref {eq:erre} is satisfied for all $ \beta < \beta_c .$   For the   Example \ref{ex:ex1} ,  condition \eqref {eq:erre} is verified. This  can be shown as in the paper by  Ferrari and Wyner (2003),  \cite{ferrariwyner}, page 23.   
  \end{rem}

\noindent
\begin{theo}\label{theomain2} {\bf (Underestimation)}
Let $ \mu$ be a translation covariant Variable-neighborhood   random field  
for which  Assumption \ref{ass:ergod} and Assumption  \ref{ass:mixing} hold. Then for any $\e
  > 0 $  there
  exists $n_0= n_0(\e , \delta, q_{min}, L) ,$ so that for any $n \geq n_0$ the probability of underestimation is bounded by
\begin{eqnarray} \label {M1} \mu \left[ \exists i \in \ll : \hat l_n
    (i) < l_i (\sigma) \right] & \le &  \exp \left( -  |
    \ll |^{1 - \e } \right) .
\end{eqnarray}
\end{theo}

\noindent
\begin{rem}
\begin{enumerate}
\item The above results  are stated for all $ n \geq n_0$ where $n_0$
  depends on the (unknown) model parameter $q_{min} $ and on the
  interaction through $L.$ 
  It is possible to write down upper bounds 
  which hold for all $n, $ but then the bounds become more  complicated
   and depend on $q_{min} $ and on $L.$  We adopted the above
   way of writing in order to state the results in a more transparent  way.
  \item Note that the trade-off between the rates of the two kind of errors (exponential
convergence for the probability of underestimation in \eqref{M1} and (basically)
polynomial convergence of the probability of overestimation in \eqref{eq:badupperbound}) 
  is a typical feature
in problems of order estimation appearing already in the simpler problem
of order estimation for Markov chains, see e.g. the papers by Finesso et al. (1996),
  \cite{finesso},
and Merhav et al. (1989), \cite{merhav}.

This represents the usual trade-off between type one and type two errors
in statistical decision problems: Overestimation means that the estimate exceeds
the true order and that we choose models that include the true
data-generating mechanism. This choice is not optimal but does only lead to a higher cost.
On the other hand underestimation leads to a restriction to lower order models that do not
describe the observed data.

So it is desirable to have an exponential control on the probability of underestimation while
keeping some polynomial control on the probability of overestimation. 
\item The definition of our estimator depends on the parameter $\delta .$ This
plays an  important  role only for the  overestimation. Namely  it 
appears in the exponent of the upper bound through the constant
$$ c( \delta ) = \frac23 \frac{2 q_{min} \delta}{e } - 2^d \log | \cA| ,$$
(see end of the proof of Lemma \ref{le:1}).
To ensure the consistency of the estimator we need to choose $\delta $
sufficiently large, depending on the one-point specification and on $q_{min} $ such
that $c(\delta) > 0 .$ Therefore, our estimator is not universal, in the sense that for
fixed $\delta $ it fails to be consistent for any random field such that $c (\delta ) < 0 .$

This problem appears even in the simpler case of order
estimation for Markov chains, see for example Finesso et al. (1996), \cite{finesso}, 
and Merhav et al.
(1989), \cite{merhav}. As pointed out by Finesso et al. (1996), \cite{finesso}, 
it is not possible to have an exponential bound on the overestimation probability of an order estimator
without rendering it inconsistent, for at least one model, for the underestimation. 
\end{enumerate}
\end{rem}

\begin{rem}
  Let us finally compare our results in the case of dimension one to the results of 
Ferrari and Wyner (2003), \cite{ferrariwyner}.  They  consider
stationary chains taking values in a finite alphabet without imposing any a priori bound on the memory. 
Hence, they are dealing with infinite trees. They overcome this difficulty by approximating the possibly
infinite memory chain by a sequence of finite range Markov chains of growing order. The price 
to pay in order to deal with these general processes is to impose geometrically $\alpha -$mixing conditions     both
for the control of the over- and the underestimation.  

In comparison to their results, to control  the underestimation, we need a slightly 
 stronger  assumption.  We  require     geometrically $ \Phi-$mixing which implies   geometrically $\alpha -$mixing.
This  is crucial 
  to obtain Theorem \ref{theomain2}.   We use Condition (\ref{eq:fin}) as sufficient condition  to obtain  the  geometrically $ \Phi-$mixing.
 
  Condition (\ref{eq:fin}), which  implies that the random field is of finite range,    could probably be relaxed.
It should be possible to deal with infinite range models, provided
one finds   other  sufficient conditions implying   mixing. 

Note however, that  mixing implies automatically the uniqueness  of
the underlying measure $\mu .$ Hence, using this kind of technique always
implies that the Dobrushin uniqueness condition (\ref{eq:erre}) must be satisfied. 
There is some hope to deal with the underestimation even
in the case of phase transition, see Remark \ref{rem:phasetransition}. 

Concerning the control of the overestimation, we are able to deal with the general
long range case without  requiring mixing. Hence we can 
do better than Ferrari and Wyner (2003), \cite{ferrariwyner}, in this aspect.
  We need to impose  the   positivity condition on the specification,  see  Assumption  \ref{ass:ergod}.
Ferrari and Wyner do only impose positivity within each step of the canonical
Markov approximation, allowing these lower bounds to tend to
zero at a certain rate.  
% Inspecting carefully our proofs, especially formula (\ref{eq:important}),
%our Assumption \ref{ass:ergod} could certainly be relaxed in the following way. Let
%$$ \Gamma_n = \min \{ p (\eta ) : \eta \in \bigcup_{\ell \le R_n} \cA^%{V_0^0 (\ell ) } \} $$
%(recall (\ref{MM2})). 
%Then there exists some $\epsilon > 0 $ such that 
%$$ \lim\inf_{n\to \infty} \frac{| \ll |^\epsilon }{R_n^d } \Gamma_n > 0 .$$ 

\end{rem}

\section{ Deviation  inequalities  for   underestimation}\label{section4}

The deviation inequalities needed for the  underestimation are based on results obtained by Dedecker (2001),    \cite {dedecker},  on
exponential inequalities for random fields.  To adapt   these results to   our model   we need      Assumption  \ref{ass:ergod}  and  Assumption \ref{ass:mixing}. 
In the next subsection we present a preliminary deviation inequality based on the results of Dedecker (2001), \cite{dedecker}, and
then give in the following subsection the deviation inequalities
needed   to control the probability of underestimation. 

\subsection{Preliminaries}
Fix $\ell >0$.  For
a given configuration $\eta \in \cA^{V^0_0 (\ell)}, $ we define
\begin{equation}\label{MM2} p(\eta)= \mu (\{ X_i = \eta_i, \; \forall
  \; i \in V^0_0 (\ell )\} ) . \end{equation}
Recall that $N_n (\eta ) $ is the total number of occurrences of $\eta$ in the observation $\sigma_{\ll }.$
Then we get the following result which is an immediate consequence of Corollary 4 of Dedecker (2001), \cite {dedecker}.

\begin{prop}\label{prop:dedecker}    Under    Assumption  \ref{ass:ergod}  and  Assumption \ref{ass:mixing}
there exists a constant $c (d, L) $ depending only on the dimension and on the range $L,$ such that 
for any configuration $\eta \in \cA^{V^0_0 (\ell)}, $
\begin{equation}\label{eq:expodedeckerbis}
\mu \left(\Big| \frac{ N_n (  \eta ) }{ |\ll |} - p  (  \eta ) \Big| \geq \delta \right) \le e^{1/e} \exp \left( 
- \frac{ c(d, L)  |\ll | \delta^2 }{  ( 2 \ell )^{2d -1}  e } 
\right) .
\end{equation}
\end{prop}   

The remainder of this section is devoted to show how this result can be obtained as a consequence of 
Corollary 4 of Dedecker (2001), \cite {dedecker}. We give this proof in detail since this shows at which extend  
Assumption \ref{ass:mixing} is needed. \\

\noindent \begin  {proof}   
For any $i,$ let 
$$ Y_i = 1_{\{ X_i^\ell = \eta  \}}  .$$
Then under $\mu ,$ $ \{ Y_i :  i \in \Z^d \}$ is a stationary random field. The associated filtration is defined
as follows. For any $\Gamma \subset \Z^d,$ let 
$$ {\cal G^\ell}_\Gamma = \sigma \{ Y_i , i \in \Gamma \} ,$$
and define the $\Phi-$mixing coefficient 
$$ \Phi ( {\cal G^\ell}_{\Gamma_1},  {\cal G^\ell}_{\Gamma_2}) = 
\sup \{ || \mu ( B | {\cal G^\ell}_{\Gamma_1}) - \mu (B) ||_\infty , B
\in {\cal G^\ell}_{\Gamma_2} \} .$$ 
Moreover, let
\begin {equation} \label {mars10}  \Phi_{\infty, 1 }^\ell(n)= \sup \{  \Phi ( {\cal G^\ell}_{\Gamma_1},  {\cal G^\ell}_{\Gamma_2}) : | \Gamma_ 2 | = 1 , 
dist (\Gamma_1, \Gamma_2 ) \geq n \} , \end {equation} 
where $dist (\Gamma_1,
\Gamma_2 )= \min \{ \|j-i\|, i \in \Gamma_1, j \in \Gamma_2 \}$.  Let
\begin{equation}\label{eq:betall}
\beta_\ell= 1 +  \sum_{n \geq 1}  \Phi_{\infty, 1 }^\ell(n) | \partial V_0(n)|.
\end{equation}
The quantity $\beta_\ell$ depends on $\ell$ through the filtration $\{{\cal G}^\ell_\Gamma, \Gamma \subset \Z^d  \}.$ To avoid confusion we warn the reader that $\beta_\ell$ defined in
\eqref {eq:betall} is a different quantity from $\beta (\ell)$ defined
in \eqref {eq:betal}, although related. 
Corollary 4 of Dedecker (2001), \cite {dedecker}, implies the following exponential
inequality
\begin{equation}\label{eq:expodedecker}
\mu \left(\Big| \frac{ N_n (  \eta ) }{ |\ll |} - p  (  \eta ) \Big| \geq \delta \right) \le e^{1/e} \exp \left( 
- \frac{ |\ll | \delta^2 }{ 4 \beta_\ell e } 
\right) .
\end{equation}
 We     estimate  $ \beta_\ell$ in  Lemma  \ref  {oct1}, stated below. 
 Then,  defining
 $ c(d,L) = \frac{1}{4 C(d,L)} $, where  $C(d,L)$ is the constant of  Lemma  \ref  {oct1},  we obtain \eqref {eq:expodedeckerbis}.  
 \end {proof}

 Assumption \ref{ass:mixing} is essential 
for proving  the following lemma.     

\noindent
\begin {lemma} \label {oct1}
  Under Assumption \ref{ass:mixing}, there exist constants $c^* = c^* (L)$ and $k = k(L)$
  such that
$$
\Phi_{\infty, 1}^\ell (n+2 \ell ) \le c^* \; |V_0^0(\ell)| e^{ - k n } , $$
and for $\beta_\ell $ defined in \eqref{eq:betall}, we have
\begin {equation} \label {feb1}  \beta_\ell   \le    
C (d, L) (2 \ell )^{2d-1} \end {equation}
 where $C(d,L)$ is a positive constant
depending on the dimension $d$ and on $L.$
\end{lemma}

\begin {proof}
For any $\Gamma \Subset \Z^d, $ let 
$$ \Gamma (\ell)= \{i \in \Z^d:  d(i, \Gamma)\le \ell \}.  $$  
We have that, whenever $|\Gamma|>1,$ 
    $$ {\cal G^\ell}_\Gamma = \sigma \{ Y_i , i \in \Gamma \}  \subset  \sigma \{ X_i , i \in \Gamma (\ell) \}
= {\cal F}_{ \Gamma (\ell) }  . $$
    When  $|\Gamma|=1$, assuming $ \Gamma= \{i \},$  
    $$ {\cal G^\ell}_\Gamma = \sigma \{ Y_i  \} \subset  \sigma \{ X_j  , j \in \Gamma (\ell)\setminus \{i\}   \} =   \sigma \{ X_j  , j \in  V_i^0 (\ell)  \}  . $$
By translational covariance, if is sufficient to set  $\Gamma_2= \{ 0 \} ,$
 $ |\Gamma_1| = \infty $ and  $dist (\Gamma_1, \Gamma_2 ) \geq n +2 \ell $, for $n \ge 1$.
Now, take $B = \{ X_0^\ell = \eta \}  $ for some fixed $\eta \in \cA^{V_0^0 (\ell ) }.$ Since $ {\cal G}^\ell_{\Gamma_1} \subset  {\cal F}_{ V_0 (n + \ell )^c } $ and $ \mu ( B | {\cal G}^\ell_{\Gamma_1} ) = 
\mu ( \mu ( B | {\cal F}_{ V_0 (n + \ell )^c } ) | {\cal G}^\ell_{\Gamma_1} ), $ in order to bound $\Phi_{\infty, 1}^\ell ( n + 2 \ell ),$ it is sufficient to bound
$$ || \mu ( B | {\cal G}^\ell_{\Gamma_1} ) - \mu (B) ||_{\infty } \le 
  || \mu ( B | {\cal F}_{ V_0 (n + \ell )^c } ) - \mu (B)  ||_{\infty }  .$$
But 
$$  \mu (B) = 
 \mu (\mu ( B | {\cal F}_{ V_0 (n + \ell )^c } )) .$$
Hence, using the specification $\gamma $ defined in \eqref{eq:3} and \eqref{eq:3.1}, by   definition \eqref {mars10} we have 
%{\bf CE: in the formula below, before the sup was not present. I added  it, now it is %simply the definition, so peraphs one could write directly the second line}
% Eva : I think it is clearer like it is written, so lets keep the first line !
\begin{equation}  \begin {split} \label {eq:diff}
& \Phi^{\ell}_{\infty, 1} (n+ 2\ell ) %= \sup \{ \sup_{B \in {\cal G}^\ell_{\Gamma_2} }   || \mu ( B | {\cal G}^\ell_{\Gamma_1} ) - \mu (B) ||_{\infty }, \G_1: |\G_1|=\infty,  dist (\Gamma_1, \Gamma_2 ) \geq n +2 \ell  \}  \le 
\\
&  \le  \sup_\omega 
\left\{ \int d \mu (\omega ') \left[ | \gamma_{ V_0 (n + \ell )} ( B | \omega ) - 
 \gamma_{ V_0 (n + \ell )} ( B | \omega' ) |\right] \right\} \\
&   \le \sup_{\omega , \omega '} \left[ |\gamma_{ V_0 (n + \ell )} ( B | \omega ) - 
 \gamma_{ V_0 (n + \ell )} ( B | \omega' )| \right] \\
&  \le \sup_{ \omega (V^0_0 (\ell)), \omega ( V_0 ( n + \ell )^c), \omega ' ( V_0 ( n + \ell )^c)}
|\gamma_{ V_0 (n + \ell )} ( \omega (V^0_0 (\ell))  | \omega ) - 
 \gamma_{ V_0 (n + \ell )} (  \omega (V^0_0 (\ell)) | \omega' ) | .
\end {split} \end{equation}
To control this last term Assumption  \ref {ass:mixing}  is essential. Indeed, we need to show that, uniformly on boundary  conditions outside  $V_0 (n + \ell )$,
  \eqref  {eq:diff} is  exponentially small in $n$.    Applying 
  Theorem 3.1.3.2 of Presutti (2009), \cite{P}, we obtain the following. There exists a function $u_{  V_0 (n+ \ell )} : \Z^d \to \R_+ $ such that
\begin{multline}
\sup_{ \omega (V^0_0 (\ell)), \omega ( V_0 ( n + \ell )^c), \omega ' ( V_0 ( n + \ell )^c)}
|\gamma_{ V_0 (n + \ell )} ( \omega (V^0_0 (\ell))  | \omega ) - 
 \gamma_{ V_0 (n + \ell )} (  \omega (V^0_0 (\ell)) | \omega' ) |
   \le \\
\sum_{i \in V^0_0 (\ell) } 
u_{  V_0 (n+ \ell )} (i). 
\end{multline}
Moreover by Corollary 3.2.5.5.~of Presutti (2009), \cite{P}, under  \eqref {eq:fin}, there exist  $c^*= c^*(L)$ and $k=k(L)$ so that 
\begin{equation}\label{eq:er2} u_{  V_0 (n + \ell )} (i) \le  c^* e^{-k d(i, V_0(n+ \ell )^c)}, i \in  V^0_0(\ell) .
\end{equation}
Therefore, we have
$$
\Phi^\ell_{\infty, 1} (n+2 \ell ) \le    c^* e^{-k n }  |V_0^0(\ell)|   ,$$
and thus
\begin{equation}     \begin {split} 
\beta_\ell & \le   | V^0_0 (2 \ell ) | + \sum_{ n \geq 2 \ell +1} | \partial V_0 (n)| \Phi^\ell_{ \infty, 1 } (n) \\
& \le   ( 4 \ell )^d  + c^*  (4 \ell )^d  \sum_{ n \geq 2 \ell +1} n^{d-1} e^{ -k (n- 2\ell)} .  \\
  \end {split}  \end{equation}
  Immediately one   gets \eqref  {feb1}. 
\end{proof}

\begin{rem}\label{rem:phasetransition}
In Proposition \ref{prop:dedecker}  we obtain an  exponential rate of convergence in the ergodic theorem.  It is very likely that to our purposes polynomial or sub-exponential  rates of  convergence will be enough. This would allow to get the control for the probability of  underestimation also 
 in the regime of phase transition.  
 %This should follow based on ideas exposed in Chazotte et al. (2007), %\cite{chazotteetal}.
This lies, however, outside the scope of the present paper. \end{rem}
 
\subsection{Deviation inequalities }
We are now able to state the deviation inequalities needed to control
the probability of underestimation. They are consequences of Proposition \ref{prop:dedecker} 
 and follow ideas of Galves and Leonardi
(2008), \cite {GL}.  Before doing so, we define for any $a \in \cA$, $\eta \in
\cA^{V^0_0 (\ell)}, $
\begin{equation}\label{LL2} p ( a | \eta ) =\frac{p( (a,\eta))}{p
    (\eta) }= \frac{\mu (\{ X_0 = a, X_i = \eta_i, \; \forall \; i
    \in V^0_0 (\ell ) \}) }{\mu (\{ X_i = \eta_i, \; \forall \; i \in
    V^0_0 (\ell ) \})}.\end{equation}
By Assumption \ref {ass:ergod} we have that 
for any given configuration $\eta \in \cA^{V^0_0 (\ell)}, $ 
$$ p(\eta)  \geq q_{min}^{(2l)^d } , $$
and
$$ p  ( a | \eta )  \geq q_{min } .$$
We are interested in configurations   having support in a ball of radius at most $L.$ Hence, writing
\begin{equation}\label{eq:encorehorrible}
\alpha_0  = \inf_{\ell \le L} \inf_{a\in \cA , \eta\in \cA^{V_0^0 (\ell) }} \{ p(a | \eta), p (\eta) \}  ,
\end{equation}
we obtain that
\begin{equation}\label{LE.1} \alpha_0  \geq q_{min}^{(2L)^d}.
\end{equation}
We define the following quantity
\begin{equation}\label{eq:deltan}
\Delta_n (  \eta ) = \sum_{ a \in \cA } \left(  \frac{N_n (\eta, a )  }{| \ll | } \log \hat p_n ( a | \eta ) 
- p ( ( \eta, a)) \log p ( a | \eta ) \right),   
\end{equation}
where $\hat p_n ( a | \eta )$ is the quantity defined in
(\ref{eq:estimatorbis}).  We obtain the following deviation
inequalities.

\noindent
\begin{cor} Let $ \mu$ be a translation covariant Variable-neighborhood   random field  
for which  Assumption \ref{ass:ergod} and Assumption \ref{ass:mixing} hold. 
Let $t > 0$, $\ell \le L$ where $L$ is given in \eqref
  {eq:fin}, let $\eta \in \cA^{V^0_0 (\ell )} $, $\hat p_n ( \cdot |
  \eta ) $ defined in (\ref{eq:estimatorbis}), $ p(\cdot | \eta) $ in
  \eqref {LL2} and let $\Delta_n ( \eta )$ as defined in \eqref
  {eq:deltan}.  Then there exists a constant $C(d,L) $ depending only on dimension and on   $L$ such that
  \begin{equation}\label{eq:ineq2}
    \mu \left(  \left|  \hat p_n ( a |  \eta) - p( a |  \eta ) \right|  > t \right) \le 2 e^{1/e} exp \left(- C(d,L) \frac{| \ll| t^2 \alpha_0  }{ 4    e } \right) , \quad \forall a \in \cA,   
  \end{equation}
  \begin{equation}\label{eq:deltan2}
    \mu \left( \left| \Delta_n (  \eta ) \right|  > t \right) \le 3 | \cA | e^{1/e} \exp \left( - C(d,L) \frac{ | \ll |  (t\wedge t^2  ) \alpha^2 _0     }{8 | \cA|^2  ( \log^2  \alpha_0     \vee 1 )  e } \right),
  \end{equation}
  where $\alpha_0 $ is given in \eqref {eq:encorehorrible} and
  estimated in \eqref {LE.1}.

\end{cor}

\begin{proof}
  Concerning \eqref {eq:ineq2} we obtain by inserting and subtracting
  the term $ \frac{N_n ( \eta, a ) }{| \ll | p ( \eta) },$
\begin{eqnarray*}
&& \left|  \hat p_n ( a |  \eta) - p ( a |  \eta) \right|   \\  
&& \le \left| \frac{N_n ( \eta, a )}{N_n ( \eta)} -  \frac{N_n ( \eta, a ) }{| \ll | p ( \eta) } \right| 
+ \left|  \frac{1 }{  p ( \eta) } \left( \frac{N_n ( \eta, a ) }{| \ll | } -  p ( ( a, \eta) ) \right)  \right| .
\end{eqnarray*}
The first term in the last expression can be upper bounded by 
\begin{eqnarray*}
&& \left| \frac{N_n ( \eta, a )}{N_n ( \eta)} -  \frac{N_n ( \eta, a ) }{| \ll | p ( \eta) } \right| 
 = N_n ( \eta, a )  \left| \frac{| \ll | p ( \eta) - N_n ( \eta) }{N_n ( \eta) | \ll | p ( \eta) }  \right| \\
& & = \frac{ N_n ( \eta, a ) }{N_n ( \eta )} \left| \frac{ p( \eta) - \frac{N_n ( \eta)}{| \ll |} }{ p  ( \eta) }  \right|  \le \left| \frac{ p ( \eta) - \frac{N_n ( \eta)}{| \ll |} }{ p ( \eta) }  \right| .  
\end{eqnarray*}
As a consequence we obtain that 
\begin{eqnarray*}
&&\mu \left(  \left|  \hat p_n ( a |  \eta) - p ( a |  \eta ) \right|  > t \right) 
\\
&&\quad \quad \quad \quad \le \mu \left(\left| p (\eta ) - \frac{N_n ( \eta)}{| \ll |} \right| > \frac t2  p ( \eta) \right) \\
&& \quad \quad \quad \quad\quad \quad \quad \quad + \mu \left(\left|\left( \frac{N_n ( \eta, a ) }{| \ll | } -  p( (\eta, a) ) \right)  \right| > \frac t2  p ( \eta) \right)  .
\end{eqnarray*}
Then, applying (\ref{eq:expodedeckerbis}), we get 
$$ \mu \left(  \left|  \hat p_n ( a | \eta) - p ( a |  \eta ) \right|  > t \right)  \le 
2 e^{\frac{1}{e}} exp \left(-  \frac{ c(d,L) | \ll | t^2 p ( \eta )^2}{ 4  ( 2 \ell)^{2d -1}  e } \right)
\le 
2 e^{\frac{1}{e}} exp \left(- \frac{c(d,L) | \ll | t^2 \alpha_0^2}{ 4 (2L)^{2d-1}  e } \right)
 . $$
Hence, writing $C(d,L) = c(d,L) / (2L)^{2d -1}, $ where $c (d,L)$ is the constant of (\ref{eq:expodedeckerbis}), assertion \eqref{eq:ineq2} follows. 
 
To show  \eqref  {eq:deltan2} we   subtract  and add  the term $ \frac{N_n ( \eta, a )}{|\ll |} \log p ( a | \eta ) $ to  $\Delta_n (\eta  )$. We obtain   
\begin{eqnarray*}
\Delta_n (\eta  ) & = &\sum_{a \in \cA}  \frac{N_n (\eta, a )}{|\ll | } \log \frac{\hat p _n ( a | \eta )}{p ( a | \eta )}  \\ 
&& + \sum_{a \in \cA} \left(\frac{N_n (\eta, a )}{|\ll |}  -p ( (\eta, a))  \right)  \log p( a | \eta ) \\
& =   & \Delta^1_n (\eta) + \Delta^2_n (\eta ) .
\end{eqnarray*}
We rewrite   $\Delta^1_n (\eta )$ in the following way and then apply the estimate (\ref{eq:divergence}):
\begin{eqnarray*}
 \Delta^1_n (\eta )  & = & \frac{N_n (\eta)}{|\ll |} \sum_{a \in \cA} \hat p_n (a| \eta)  \log \frac{\hat p _n ( a | \eta )}{p ( a | \eta )} \\
&&  \le \frac{N_n (\eta)}{|\ll |} \sum_{a \in \cA} \frac{\left( \hat p_n (a |\eta) - p (a |\eta) \right)^2 }{p (a |\eta) } \\
&& \le  \sum_{a \in \cA} \frac{\left( \hat p_n (a |\eta) - p (a |\eta) \right)^2 }{p (a |\eta) }  \\
&& \le  \sum_{a \in \cA} \frac{\left( \hat p_n (a |\eta) - p (a |\eta) \right)^2 }{ \alpha_0} .  
\end{eqnarray*}
Therefore  
\begin{eqnarray}\label{eq:horrible}
&& \mu \left( | \Delta^1_n (\eta)  | > \frac t2 \right) \nonumber \\
&& \le \sum_{a \in \cA} \mu \left( \left( \hat p_n (a |\eta) - p (a |\eta) \right)^2 > \frac{1}{| \cA| }\frac t2 \alpha_0    
\right) \nonumber \\
&& \le  2 | \cA | e^{1/e} \exp \left( - C(d,L) \frac{ | \ll |  t \;   \alpha_0^2   }{8  | \cA|   e } \right), 
\end{eqnarray}
by (\ref{eq:ineq2}).
We get for the second term   
\begin{eqnarray}\label{eq:horrible2}
&& \mu \left( | \Delta^2_n (\eta )  | > \frac t2 \right) \nonumber \\
&& \le \sum_{a \in \cA} \mu \left( \left|\frac{N_n (\eta, a )}{|\ll |}  - p  ( (\eta, a))  \right| >    \frac{1}{| \cA| }\frac t2
\frac{1}{|\log   \alpha_0    |}
\right) \nonumber \\
&& \le | \cA| e^{1/e} \exp \left( - C( d,L) \frac{| \ll | t^2}{ 4 | \cA |^2 \log^2  \alpha_0     e  }\right) ,  
\end{eqnarray}
by  (\ref{eq:expodedeckerbis}).
This finishes the proof. 
\end{proof}

\section{Deviation inequalities for overestimation}\label{section:5}
In order to control the probability of overestimation we do not need as strong assumptions as 
for the control of the probability of underestimation. Indeed, we can avoid to impose Assumption \ref{ass:mixing}.
We mimic the method used  by Csisz\`ar and Talata (2006), \cite{CT}, see Proposition 3.1 and Lemma 3.3   of  their paper. 
Their results are typicality results and they obtain the 
almost sure convergence of the empirical probabilities to the theoretical ones. We follow the way indicated
by Csisz\`ar and Talata (2006), \cite{CT}, but we  quantify the errors and obtain in this way precise deviation inequalities.   We will need  only Assumption \ref {ass:ergod}.

We partition the region $ \bar \Lambda_n$ by intersecting it with a
sub lattice of $\Z^d$ such that the distance between sites in the sub
lattice is $4 R_n+1$.  More precisely, let
 $$ \ll^k = \{ j \in \ll, \; j = k + (4R_n +1) l , l \in \Z^ d \}, \quad \| k \| \le 2 R_n. $$
For any $\ell \le R_n$ and any fixed configuration $\eta \in \cA^{V_0^0 (\ell)},$ let 
$$ N_n^k (\eta ) = \sum_{ j \in \ll^k } 1_{ \{ X_j^ \ell = \eta \}} $$
be the number of occurrences of $\eta$ in the sample having center in $\ll^k .$ In the same way we denote 
$$ N_n^k (\eta , a) = \sum_{ j \in \ll^k } 1_{ \{  X_j^ \ell = \eta , X_j = a \}} .  $$
Note that we have 
$$  N_n (\eta ) = \sum_{ k : \| k \| \le 2 R_n}  N_n^k (\eta ), \; N_n (\eta, a ) = \sum_{ k : \| k \| \le 2 R_n}  N_n^k (\eta , a  ) .$$
Let 
\begin{equation}\label{LL3} {\cal A} (n,\ell, k)= \{ \frac32 \log
  N_n^k (\eta ) \geq \log | \ll |, \mbox{ for all } \eta \in
  \cA^{V_0^0 (\ell)} \mbox{ s.t.  } \ell \geq l_0 (\eta)  \}\end{equation}
and 
\begin{equation}\label{eq:rest2} {\cal B} (n ,\ell ) = \bigcap_{ k :
    \| k \| \le 2 R_n} {\cal A} (n , \ell, k) .
\end{equation}  
The probabilities $ \mu ( {\cal A} (n,\ell, k))$ and $ \mu
({\cal B} (n ,\ell ))$ can be immediately obtained by Lemma \ref
{le:g} given at the end of this section.  Recall the definition of
$\hat p_n $ in (\ref{eq:estimator}).

\noindent
\begin{theo}\label{theo1}  For any 
\begin{equation}\label{eq:choiceofdelta}
 \delta > 2^d \log | \cA | \frac{3 e}{4 q_{min} },
\end{equation}
there exist a positive constant $c(\delta)= c(\delta, q_{\min}) $ and $n_0$ (not depending on $q_{min}$ nor on $\delta $) such
that for all $ n \geq n_0 ,$ for any $\ell \le R_n,$ 
\begin{eqnarray}
&& \mu   \Bigg [ \mbox{ $\exists $ } \eta \in \cA^{V_0^0 (\ell)} ,  \ell \geq l_0 (\eta ) :   \left | \hat p_n (a| \eta ) - \gamma_{\{0 \}} (a| \eta  )\right | >   \sqrt {\kappa (\delta)  \gamma_{\{0 \}} (a| \eta  ) 
\frac{  
{ \log N_n (\eta )  }}{ N_n (\eta )}  } \; \; , \nonumber \\
&& \quad \quad \quad \quad \quad \quad \quad \quad   
  {\cal B} (n ,  \ell ) \Bigg ]  \nonumber \\
&&  \quad \quad \quad \quad   \le 4 \,   (4 R_n + 1 )^d 
  \exp \left(  -c(\delta) \sqrt{   \log | \ll |  }  \right),
\end{eqnarray}
where $\kappa (\delta)>0 $ is as in \eqref{eq:kappa}. 
\end {theo}
 
\noindent
The main ingredient to  prove Theorem \ref{theo1} is the following lemma.
\vskip0.5cm

\noindent
\begin{lemma} \label {le:1} 
%Under the assumptions of Theorem \ref{theo1},
  For any $\delta$ as in \eqref {eq:choiceofdelta} there exist $n_0 $
  (not depending on $q_{min}$ nor on $\delta $) and a positive
  constant $c(\delta)= c(\delta, q_{\min}),$ such that for all $n \geq n_0,$ for any $\ell \le R_n , $
 \begin{eqnarray}\label{eq:66} && \mu   \Bigg[ \mbox{ $\exists $  }  \eta \in \cA^{V_0^0 (\ell)} ,  \ell \geq l_0 (\eta ) :  \left |  \frac { N_n^k (\eta , a) } {   N_n^k (\eta ) } - \gamma_{\{0\}} ( a|  \eta  )  \right |  \ge  \sqrt  {\delta  \gamma_{\{0\}} ( a|  \eta  )  \frac{(\log N_n^k (\eta )^{\frac12}}{ N_n^k (\eta )  },   
   }    \nonumber \\
   && \quad \quad \quad \quad \quad \quad \quad \quad    \;  {\cal A} (n , \ell, k)  \Bigg]  
   \nonumber \\
   && \le
   4    \exp \left(  -c (\delta )   \sqrt{   \log | \ll |  }  \right) .
 \end{eqnarray} 
 \end{lemma}

\begin {proof} 
  Fix $ \eta \in \cA^{V_0^0 (\ell)} $ with $ \ell \geq l_0 (\eta )$ and  set  $ \gamma (a)
  = \gamma_{\{0\} } ( a| \eta ).$ Recall that $\gamma (a) \geq q_{min}
  .$ We first provide an upper bound for fixed $\eta $ of
$$ \mu  \left [ \left | N_n^k (\eta , a)    -    N_n^k (\eta ) \gamma (a)  \right | 
  \ge \sqrt { \delta \gamma (a) N_n^k (\eta ) (\log N_n^k (\eta
    ))^{1/2} } \; , {\cal A} (n, \ell, k)
\right ] .
$$
By definition
\begin{equation}\label{eq:7b } N_n^k (\eta , a) - N_n^k (\eta )
  \gamma_{\{0\}} ( a| \eta ) = \sum_{ j \in \ll^k } 1_{ \{ X_j^ \ell
    = \eta \}} \left [ 1_{ \{ X_j = a \} } - \gamma_{\{0\}} ( a|
    \eta ) \right ].  \end {equation}
We order in some arbitrary way  the points 
$$ \{  j \in \ll^k,   X_j^ \ell = \eta   \} =   \{ j_l ,    1 \le l \le  N_n^k (\eta )   \}.  $$  
Define
$$ Z_l= \left [  1_{ \{ X_{ j_l} = a \} } -  \gamma_{\{0\}} ( a|  \eta  )  \right ], \quad l=1, \dots  N_n^k (\eta ).$$
The random variables $ \{  Z_l, \;  l=1, \dots  N_n^k (\eta )\}$ are identically  distributed random variables  with   
  mean zero and conditionally independent, i.e.     for $i \neq j$,  $0 \le | z_i | \le 1 $, $0 \le |z_j| \le 1 $ 
\begin{eqnarray*}
&& \mu   \left [ Z_i =z_i, Z_j=z_j |    \omega ( \ll \setminus \cup_{j \in \ll^k }  V_j (\ell) )\right ]  =  \mu  \left [ Z_i =z_i |    \omega ( \ll \setminus \cup_{j \in \ll^k }  V_j (\ell) )\right ] \\
&& \quad \quad \quad\quad\quad\quad \quad \quad \quad\quad\quad\quad \quad \quad \quad\quad\quad\quad \cdot \mu  \left [ Z_j=z_j  |    \omega ( \ll \setminus \cup_{j \in \ll^k }  V_j (\ell) )\right ].
\end{eqnarray*}
Take an independent copy $\{ Z'_i, i \geq 1 \} $ of i.i.d. $\! \!$ random variables, having the same distribution as $Z_1,$ independent of $X.$ 
Then for $i  > N_n^k (\eta )$ we let  $Z_i = Z'_{i -  N_n^k (\eta )} . $ The important point of this definition is that in this
way, the sequence of random variables $ Z_1, Z_2 , \ldots $ is independent of $N_n^k (\eta  ) .$ Define partial sums 
 $$ S_N= \sum_{j=1}^N Z_j, \quad  S_N^* = \max \{ S_j; j \le N \} .$$ 
 These are still independent of $N_n^k (\eta  ) .$
 We write the quantity in \eqref {eq:7b } as 
 \begin{equation}\label{eq:7c } N_n^k (\eta , a)    -    N_n^k (\eta ) \gamma_{\{0\}} ( a| \eta  ) = S_{ N_n^k (\eta )} \le  S^*_{ N_n^k (\eta )}.
  \end {equation}  
 We now use arguments similar to those in the proof of Lemma 3.3 of Csisz\`ar and Talata (2006), \cite{CT}. In the following, 
 $$\tilde \mu = \mu  ( \cdot | \omega ( \ll \setminus \cup_{j \in \ll^k }  V_j (\ell) ) ) $$ 
 denotes always conditional 
 probability when conditioning with respect to $\omega ( \ll \setminus \cup_{j \in \ll^k }  V_j (\ell) ).$ Then,
  \begin{eqnarray}\label{eq:8}
  && \tilde \mu  \left [  S^*_{ N_n^k (\eta )} 
   \ge  \sqrt  {   \delta   \gamma (a)   N_n^k (\eta )   (\log N_n^k (\eta ))^{1/2}
 } \quad  ,   {\cal A} (n,  \ell , k) 
    \right ]  \le 
  \\
 &&  \sum_{j \in \N} \tilde \mu  \! \! \left [ S^*_{ N_n^k (\eta )} 
  \! \!  \ge \! \!  \sqrt  { \delta    \gamma (a)   N_n^k (\eta )   (\log N_n^k (\eta ))^{1/2}
 } ;  e^j \! \! < N_n^k (\eta ) \le e^{j +1} \!   ,  {\cal A}  (n,  \ell, k)
    \right ] 
       . \nonumber
\end {eqnarray}
Note that on ${\cal A} (n,   \ell ,k ) \cap \{e^j < N_n^k (\eta ) \le e^{j +1}  \} ,$  see \eqref {LL3},  
since $\log N_n^k (\eta ) \le \log |\ll |,$ 
$$ j < \log | \ll | \le   \frac32 (j+1).$$

\noindent
Hence by independence of $\{ S^*_N , N \geq 1 \} $ and $N_n^k (\eta ) ,$ the last expression of \eqref{eq:8} 
can be bounded from above as follows. 
\begin{equation}\label{eq:jrestriction}
  \sum_{ j :  j < \log | \ll | \le  \frac32 (j+1)  } \tilde \mu \left [ S^*_{ e^{j+1}}  
  \ge  \sqrt  {  \delta  \gamma (a)     e^j  \sqrt{j}} \right] \le
 \sum_{ j :   \log | \ll | \le  \frac32 (j+1)  } \tilde \mu \left [ S^*_{ e^{j+1}}  
  \ge  \sqrt  {  \delta   \gamma (a)     e^j  \sqrt{j}} \right]  .
\end{equation}  
Now, Bernstein's inequality, see Lemma \ref{le:bernstein}, yields
$$ \tilde \mu \left[  S^*_N \ge c \right] \le 2 \exp \left( - \frac{2 c^2}{N  } \right)  . $$
This gives 
$$ \tilde \mu   \left [ S^*_{ e^{j+1}}  
  \ge  \sqrt  { \delta     \gamma (a)    e^j  \sqrt{j}} \right]  \le 2 \exp \left(  -\frac{2 q_{min}  }{e} \delta \sqrt{j} \right)    .$$ 
Taking into account that 
$$\int_{\sqrt a} ^\infty e^{-b y } y  dy=    \frac 1 b e^{-b {\sqrt a} } {\sqrt a} + \frac 1 {b^2}e^{-b {\sqrt a} },   $$
setting   $b = 2 q_{min} \delta / e  $ and $ a =  \frac23 \log |\ll | - 1,$ 
one can upper bound the sum over $j$   in \eqref {eq:jrestriction} obtaining 
\begin{eqnarray*}
&&  \tilde \mu  \left [  S^*_{ N_n^k (\eta )} 
   \ge  \sqrt  { \delta    \gamma (a)    N_n^k (\eta )   (\log N_n^k (\eta ))^{1/2}
 } \quad ,  {\cal A}  (n,  \ell, k) 
    \right ]  \\
    &&  \quad \quad \quad \quad \le   \frac{2e}{\delta q_{min}}\left( \sqrt{\frac23 \log |\ll | - 1 } + \frac{e}{2 \delta q_{min}} \right) 
       \exp \left(  -\frac{2 q_{min} \delta   }{e} \sqrt{   \frac23  \log | \ll | - 1  }  \right)  \\ 
    && \quad \quad \quad \quad \le
     4  \left( \sqrt{\frac23 \log |\ll | - 1 } + 1  \right) \exp \left(  -\frac{2 q_{min} \delta   }{e} \sqrt{   \frac23  \log | \ll | - 1  }  \right)
  ,
\end{eqnarray*}
since by the choice of $\delta $ in (\ref{eq:choiceofdelta}), $\frac{e}{2 \delta q_{min}} \le 1.$

Now, there exists $n_0 $ (not depending on $q_{min}$ nor on $\delta $) 
such that for all $ n \geq  n_0,$ this last upper bound can be replaced by  
$$ \left( \sqrt{\frac23 \log |\ll | - 1 } + 1  \right) \exp \left(  -\frac{2 q_{min} \delta   }{e} \sqrt{   \frac23  \log | \ll | - 1  }  \right)  \le \exp \left(  -  \frac23 \frac{2 q_{min} \delta   }{e} \sqrt{   \log | \ll |   }  \right) .   $$
\noindent
This upper bound also holds for the non-conditioned probability $\mu  .$ 
Finally, in order to get the result uniformly over all possible configurations $\eta $ having $ l_0 (\eta ) \le \ell ,$ we need to sum over all 
possible choices of patterns $ \eta  .$ This gives, by definition of $R_n,$  
$$ |\cA |^{|V_0^0 (\ell)| }  =
 |\cA |^{(2 \ell )^d } \le |\cA |^{(2 R_n  )^d } = e^{ 2^d  \log | \cA|  \sqrt{ \log | \ll |}   }$$
terms. Thus we can conclude that for all $n \geq n_0,$  taking  $\delta $ as in \eqref {eq:choiceofdelta} we have
\begin{eqnarray*}
&& \mu   \Bigg[ \mbox{ $\exists $ }  \eta \in \cA^{V_0^0 (\ell)} , \ell \geq l_0 (\eta) :  \left |  \frac { N_n^k (\eta , a) } {   N_n^k (\eta ) } - \gamma_{\{0\}} ( a|  \eta  )  \right |  \ge  \sqrt  {\delta  \gamma_{\{0\}} ( a|  \eta  )  \frac{(\log N_n^k (\eta ))^{\frac12}}{ N_n^k (\eta )  }  
 }    \nonumber \\
&& \quad \quad \quad \quad \quad \quad \quad \quad  \quad \quad \quad  \quad \quad \quad \quad \quad \quad \quad \quad    \; , {\cal A} (n , \ell, k)  \Bigg]  
\nonumber \\
&& \le  4  e^{ 2^d  \log | \cA|  \sqrt{ \log | \ll |}   }  \exp \left(  -  \frac23 \frac{2 q_{min} \delta   }{e} \sqrt{   \log | \ll |   }  \right) \\
&& = 4 \exp \left( - c (\delta )  \sqrt{ \log | \ll |}    \right) , 
\end{eqnarray*}    
where $c  (\delta ) =  \frac23 \frac{2 q_{min} \delta   }{e} - 2^d  \log | \cA| > 0 .$ This concludes the proof. 
\end {proof}

\noindent
We are now able to give the proof of Theorem \ref{theo1}.

\noindent
{\bf Proof of Theorem \ref{theo1}}
Fix $\eta \in \cA^{V_0^0 (\ell)} $ with $ \ell \geq l_0 (\eta ) ,$  let $\gamma (a) = \gamma_{\{0\}} ( a| \eta  )$, $\delta$ as in \eqref{eq:choiceofdelta} and denote by 
 $$ E_n (\eta ) = \bigcap_{ k : \| k \| \le 2 R_n } \left\{ \left |  \frac { N_n^k (\eta , a) } {   N_n^k (\eta ) } - \gamma (a)  \right |  \le  \sqrt  { \delta \gamma (a)  [   N_n^k (\eta )]^{- 1  } \sqrt{ \log N_n^k (\eta  ) }  
 }   \right\} . $$
 Then on $E_n (\eta ) , $ using Jensen's inequality, the definition of $R_n$  and $N_n^k ( \eta )  \le N_n ( \eta ) , $ 
 \begin{eqnarray}
&&  \left | \hat p_n (a| \eta ) - \gamma_{\{0\}} (a|  \eta  )\right |  \nonumber \\
&& \le \sum_{ k : \| k \| \le 2 R_n }  \left | \frac{ N_n^ k(\eta , a)}{ N_n^ k( \eta )  }  - \gamma_{\{0\}} (a|  \eta  )\right | \cdot \frac{ N_n^ k(\eta ) }{ N_n( \eta ) }\nonumber \\
&& \le \sum_{ k : \| k \| \le 2 R_n }  \sqrt{ \delta  \gamma_{\{0\}} (a| \eta  ) \frac{  \sqrt{ \log N_n^k (\eta  ) }}{N_n^k (\eta )}
 } \cdot \frac{ N_n^k ( \eta ) }{ N_n(\eta) } \nonumber \\
 && \le \sqrt{  \sum_{ k : \| k \| \le 2 R_n }    \delta  \gamma_{\{0\}} (a| \eta  )  \frac{ \sqrt{ \log N_n^k (\eta ) } }{ N_n(\eta ) }    } \nonumber \\
 && \le  \frac{( 4 R_n + 1 )^ { d/2} \delta^{1/2} \gamma_{\{0\}} (a| \eta  )^{1/2}  [ \log N_n (\eta  ) ]^{1/4} }{[ N_n( \eta ) ]^{1/2}} \nonumber \\
 && \le 5^{d/2}  \left(   \log | \ll | \right)^{\frac14} \delta^{1/2}  \gamma_{\{0\}} (a| \eta  )^{1/2}  \frac{[\log N_n (\eta ) ]^{1/4} }{ [N_n( \eta )]^{1/2} } . 
 \end{eqnarray}
 On $ \{  \log | \ll | \le \frac 32  \log N_n( \eta )   \}, $ this last expression can be bounded from above by 
 $$ 5^{d/2} ( \frac 32)^{ 1/4} \sqrt{ \delta  \gamma_{\{0\}} (a| \eta  )  \frac{ \log N_n (\eta  )  }{ N_n( \eta)  }} 
 =  \sqrt{   \kappa (\delta) \gamma_{\{0\}} (a| \eta  )  \frac{ \log N_n (\eta  )  }{ N_n( \eta)  }}, $$
where $\kappa (\delta)$ is chosen as in \eqref {eq:kappa}.
Hence we get 
 \begin{eqnarray}\label{eq:firststep}
&&  \mu  \left[ \exists \;\eta \in \cA^{V_0^0 (\ell)}, \ell \geq l_0 (\eta )    : \left | \hat p_n (a| \eta ) - \gamma_{\{0\}} (a|  \eta  )\right | >   \sqrt{ \kappa (\delta) \gamma_{\{0\}} (a| \eta  )    \frac{ \log N_n (\eta )  }{ N_n( \eta ) }}
\; ,  {\cal B}(n,   \ell)  \right ] \nonumber \\
&& \quad \quad \quad \quad \quad \quad \quad \quad \quad \quad \le   
\mu \left[ \bigcup_{\eta \in \cA^{V_0^0 (\ell)}  : \;  \ell \geq l_0 (\eta)  }  E_n(\eta )^c \; ,  {\cal B}(n,   \ell)  \right] . 
 \end{eqnarray}
 But
\begin{multline*}
  \bigcup_{\eta \in \cA^{V_0^0 (\ell)} :  \;  \ell \geq l_0 (\eta) }  E_n(\eta )^c \\
  =  \bigcup_{ k : \| k \| \le 2 R_n } \bigcup_{ \eta \in \cA^{V_0^0 (\ell)}: \;  \ell \geq l_0 (\eta)  }  
 \left\{ \left |  \frac { N_n^k (\eta , a) } {   N_n^k (\eta)   } - \gamma (a)  \right | >  \sqrt  { \delta \gamma (a) \frac{  \sqrt{ \log N_n^k (\eta ) }}{    N_n^k (\eta )} }    \right\}, 
\end{multline*} 
therefore applying  Lemma \ref{le:1} we can finally upper bound 
 $$    \mu  \left[ \bigcup_{\eta :  \;  \ell \geq l_0 (\eta)}  E_n(\eta)^c \; ,  {\cal B}(n,  \ell)  \right]  
   \le 4  (4 R_n + 1 )^d      \exp \left(  - c (\delta )   \sqrt{  \log | \ll |  }  \right)    ,  $$
for all $ n \geq n_0 .$ This finishes the proof.  
{\hfill $\bullet$ \vspace{0.25cm}}

The following lemma gives conditions ensuring that $\mu  ( {\cal B}(n , \ell)^c) $ converges to $0$ by giving
the precise rate of convergence. 

\noindent 
\begin{lemma} \label {le:g}
For  any $ 0< \e_1 <1$, $0< \e_2<1$, and for any positive $C_1$ and $ C_2$ there exists $n_0 =n_0(q_{min},  \min (\e_1,\e_2), \min (C_1,  C_2) )$ so that for $n \ge n_0$ and for any $\ell \le R_n,$  
  we have 
  \begin{equation}\label{eq:333} \mu \left( \exists \, \eta \in
      \cA^{V_0^0(\ell)} ,   \ell \geq l_0 (\eta): N_n^k (\eta ) < C_1 | \ll |^{1-\e_1} \right)
    \le \exp \left(  -  C_2 |\ll |^{ 1- \e_2} \right)  . \end {equation} 
  \end{lemma}
   
\begin {proof} 
Fix some $\eta $ with $  l_0 (\eta) \le \ell \le R_n .$ Then $ \{ 1_{\eta} (X_j^\ell )  ,  \; j \in \ll^k \}$  is a collection of conditional independent random variables, conditioned on  fixing the configuration $ \omega  ( \ll \setminus \cup_{j \in \ll^k }  V_j (\ell))$.  By {Assumption}  \ref {ass:ergod},  we have that  
   $$ \mu  (X_j^ \ell = \eta  ) \ge q_{min}^{(2 \ell)^d } .$$ 
Here we have used that $ | V^0_0(\ell) |=  (2\ell )^d$.
Then a conditional version of the Hoeffding inequality, see for example Lemma A3 in Csisz\`ar and Talata (2006), \cite{CT}, yields
 \begin{equation}\label{eq:1a}
  \mu  \left [ \frac {   N_n^k (\eta  ) } { |\ll^k| }  < \frac 12 q_{min}^{(2 \ell)^d} |  \omega  ( \ll \setminus \cup_{j \in \ll^k }  V_j (\ell) )\right ]   \le   e^{ -  |\ll^k|  \frac {q_{min}^{(2\ell)^d }} {16}  }.  \end {equation}
As a consequence, we obtain also for the unconditioned probability,
  \begin{equation}\label{eq:2first}
  \mu \left [ \frac {   N_n^k (\eta ) } { |\ll^k| }  < \frac 12 q_{min}^{(2\ell)^d} \right ]   \le   e^{ -  |\ll^k|  \frac {q_{min}^{(2\ell)^d }} {16}  },\end {equation}
 and thus 
 \begin{equation}\label{eq:2}
 \mu  \left [ \exists \, \eta \in \cA^{V_0^0(\ell)} , \;  \ell \geq l_0 (\eta) \, : \frac {  N_n^k (\eta  ) } { |\ll^k| }  <  \frac 12 q_{min}^{(2\ell)^d} \right ]   \le  | \cA|^{ (2 R_n)^d}  e^{ -  |\ll^k|  \frac {q_{min}^{(2\ell)^d }} {16}  }.
 \end{equation}
 To obtain \eqref {eq:333}  we need to compare
 $ | \ll |$ to $  |\ll^k|$.  By construction we have for $n $ sufficiently large, 
 \begin{equation}\label{eq:4}
 |\ll^k|  \ge   \frac {|\ll |} { (4R_n +1)^d}  \ge  \frac {|\ll |} { (5R_n)^d  } .\end {equation}
This and \eqref  {eq:2}     imply that
\begin{equation}\label{eq:2a}
\mu   \left [  \exists \, \eta \in  \cA^{V_0^0(\ell)} , \;  \ell \geq l_0 (\eta) \, :   N_n^k (\eta  )   < \frac 12 q_{min}^{(2\ell)^d}     \frac {|\ll |} { (5R_n)^d }   \right ]   \le   | \cA|^{ (2 R_n)^d}   e^{ -  \frac {|\ll |} { (5R_n)^d } \frac {q_{min}^{(2\ell)^d }} {16}  }.  \end {equation}
But 
\begin{equation}\label{eq:important}
  q_{min}^{(2\ell)^d}     \frac {|\ll |} { (5R_n)^d }   \ge      |\ll |    
      \frac { q_{min}^{(2R_n)^d}  } { (5R_n)^d }. 
\end{equation}      
By  the definition of $R_n$  in \eqref {eq:rn},  $R_n^d =  \sqrt{\log |\ll | }.$ 
Thus for any $C>0$ and for any  $\epsilon>0$  there exists $n_0= n_0(q_{min}, \epsilon, C)$  so that for $n \ge n_0 , $
 $$     |\ll |^\epsilon  \frac { q_{min}^{(2R_n)^d}  } { (5R_n)^d } =   |\ll |^\epsilon 
 \frac { e^{2^d \sqrt{\log |\ll | } \log q_{min} } }  { 5^d  \sqrt{\log |\ll | }} \ge C  . $$
 This and \eqref  {eq:2a} imply  that for any $\e_1>0$ and $\e_2>0$, positive $C_1$ and $ C_2,$ 
 for  $n\ge n_0,  $ we have  
  \begin{equation}\label{eq:2b}
  \mu  \left [ \exists \, \eta \in  \cA^{V_0^0(\ell)} , \;  \ell \geq l_0 (\eta) \, :  N_n^u (\eta  )   <  C_1    |\ll |^{ 1-\e_1}  \right ]   \le   | \cA|^{ (2 R_n)^d}   e^{ - 2  C_2    |\ll |^{ 1-\e_2} }.  \end {equation}
  Finally, note that for $n \geq n_0 ,$
  $$ | \cA|^{ (2 R_n)^d} = e^{ 2^d \log | \cA | \sqrt{ \log | \ll | } } \le     e^{  C_2    |\ll |^{ 1-\e_2} } .$$ 
 Thus we have proved the lemma.
  \end{proof}

\section{Proof of Theorem \ref{theomain1}}
We show the probability of overestimation \eqref {eq:badupperbound}. Recall the  definition of the set  $  {\cal B} (n,  R_n) $  given in  \eqref{eq:rest2}. Clearly,
\begin {equation}  \label {ro1a}  \begin{split}
    \mu ( \exists \; i \in \ll : \hat l_n (i) > l_i (\sigma ) )  &\le    \mu ( {\cal B} (n,  R_n)^c ) \\
    &  + \mu ( \exists \; i \in \ll : \hat l_n (i) > l_i( \sigma ) , {\cal B} (n,  R_n) ). \\
   % & \le  | \bar \Lambda_n| \bar d ( \mu, \mu^{R_n}) +  \mu^{R_n} ( ({\cal B} (n,  R_n))^c %) \\
   % &  + \mu^{R_n} ( \exists \; x \in \ll : \hat l_n (x) > l_x( \sigma ) , {\cal B} (n,  R_n) ) .
  \end {split} \end {equation}  The first term is estimated by
Lemma \ref{le:g}, choosing $\e_1 = \frac13 , \e_2 = \e, $ $C_1 = 1, C_2 =2.$ This yields
 $$  \mu  ( ({\cal B} (n,  R_n))^c ) \le ( 4 R_n +1)^d e^{ - 2  | \ll |^{1 - \e } } ,$$
 for all $ n \geq n_0 $ where $n_0 $ depends on the choices $ \e_1=
 \frac 13 , \e_2 = \e , C_1= 1,C_2 = 2$ and   $q_{min}$.  Since 
 $$ ( 4 R_n +1)^d  \le C(d) \sqrt{ \log | \ll |  } \le C(d) e^{    | \ll |^{1 - \e} }, $$
 eventually, we have that for all $ n \geq n_0$
\begin {equation}  \label {ro1b}  \mu  ( ({\cal B} (n,  R_n))^c ) \le C(d) e^{ -  | \ll |^{1 - \e } }. \end {equation}
 
We now study the last term of \eqref {ro1a}. We are interested in the event $ \{ \hat l_n (i) = \ell > l_i (\sigma ) \} . $
 Note that  $\ell > l_i( \sigma )$ implies that for any $j$ such that $ X_j^{\ell-1} = \sigma_i^{{\ell}-1},$ necessarily
 $ l_j (\sigma ) = l_i (\sigma ) \le \ell - 1 $ and 
 as a consequence $ \gamma_j (\cdot | X_j^{\ell -1}) = \gamma_i (\cdot | \sigma_i^{{\ell}-1} ).$ 
 
 Hence, for any $ \ell > l_i (\sigma ),$ we
 have, by \eqref {eq:loglikelihood4}
\begin{eqnarray}\label{eq:ub1}
\log L_n (i,\ell)&= &    \sum_{j : \;  X_j^{\ell-1} = \sigma_i^{{\ell}-1}  }  \frac{1}{N_n (X_j^\ell) }\left [ \log  MPL_n (j,\ell) - \log  MPL_n (i,\ell-1)\right ] \nonumber \\
& \le &   \sum_{j : \;  X_j^{\ell-1} = \sigma_i^{{\ell}-1}  } \frac{1}{N_n (X_j^\ell) } \left [ \log  MPL_n (j,\ell) - \log  PL_n^{ (i,\ell-1)} ( \gamma_i (\cdot | \sigma_i^{{\ell}-1} ) ) \right ]\nonumber \\
& = &  \sum_{j : \;  X_j^{\ell-1} = \sigma _i^{{\ell}-1}  } \frac{1}{N_n (X_j^\ell) } \sum_{a \in A} \left( N_n( X_j^\ell , a) \log \left[\frac{ \hat{p}_n ( a | X_j^\ell) }{ \gamma_i ( a | \sigma_i^{{\ell}-1} ) } \right] \right) \nonumber \\
& = &  \sum_{j : \;  X_j^{\ell-1} = \sigma_i^{{\ell}-1}  }  \sum_{a \in A} \left(  \hat{p}_n ( a | X_j^\ell) \log \left[\frac{ \hat{p}_n ( a | X_j^\ell) }{\gamma_i (a | \sigma_i^{{\ell}-1} ) } \right] \right) \nonumber \\
& \le &  \sum_{j : \;  X_j^{\ell-1} = \sigma_i^{{\ell}-1}  }  \sum_{a \in A} \frac{\left( \hat{p}_n ( a | X_j^\ell) - \gamma_i (a | \sigma_i^{{\ell}-1} ) \right)^2
}{\gamma_i (a | \sigma_i^{{\ell}-1} )} \nonumber \\
&=&\sum_{j : \;  X_j^{\ell-1} = \sigma_i^{{\ell}-1}  }  \sum_{a \in A} \frac{\left( \hat{p}_n ( a | X_j^\ell) - \gamma_j (a | c_j (\sigma) ) \right)^2
}{\gamma_j (a | c_j (\sigma) )}.
\end{eqnarray}
We  used that for any two probability distributions $P$ and $Q$ on $A, $
\begin{equation}\label{eq:divergence}
 \sum_a P(a) \log \frac{P(a)}{Q(a)} \le \sum_a \frac{(P(a) - Q(a))^2}{Q(a)} 
\end{equation} 
 (see Csisz\`ar and Talata \cite{CTbis}), 
and in the last line the fact that $ \{ X_j^{\ell-1} = \sigma_i^{{\ell}-1} ,  \ell > l_i (\sigma ) \} $ implies $ \gamma_j (a | c_j (\sigma) ) = \gamma_i (a|  \sigma_i^{{\ell}-1} ) .$ Hence,  writing  for short $\gamma_{j}
 (\cdot) = \gamma_{j} ( \cdot | c_{j} (\sigma))$,  define 
\begin{multline*}
 E_\ell =\\
    \left\{ \forall \, j \in \ll ,  l_j (\sigma ) < \ell, \forall a \in \cA : \left | \hat p_n (a| \sigma_j^{\ell}) - \gamma_{j} (a)\right | \le     \sqrt {\kappa (\delta)  \gamma_{j } (a) 
\frac{  
{ \log N_n (\sigma_j^{\ell})  }}{ N_n (\sigma_j^{\ell})}  } \right\} ,
\end{multline*}
where $ \delta$  is as in \eqref {eq:choiceofdelta}  and  $\kappa (\delta) $  is  defined in \eqref {eq:kappa}. 
Then on $E_\ell ,$ \eqref{eq:ub1} can be bounded uniformly in $i \in \ll $ from above by 
\begin{multline*}
   \sum_{j : \;  X_j^{\ell-1} = \sigma_i^{{\ell}-1}  }  \sum_{a \in \cA} \kappa (\delta ) \frac{\log N_n ( X_j^\ell )  }{ N_n (X_j^\ell)} 
= \sum_{ v \in \cA^{\partial V_0 (\ell) } } N_n( \sigma_i^{\ell -1} v ) \sum_{a \in \cA} \kappa (\delta ) \frac{\log N_n ( \sigma_i^{\ell -1} v )  }{ N_n (\sigma_i^{\ell -1} v)} \\
\le \kappa (\delta) | \cA| \; |\cA|^{ |\partial V_0(\ell)| } \log |
\ll | .
\end{multline*} 
Hence,  on $\bigcap_{\ell = 1 }^{ R_n} E_\ell ,$ for all $i \in \ll : $ We have for all  $ \ell > l_i (\sigma) , $ 
$$  \log L_n (i,\ell) \le  \kappa(\delta) | \cA|  |\cA|^{| \partial V_0(\ell) |   }  \log | \ll |  = pen (\ell , n) .$$
This implies that $ \hat l_n (i ) \le l_i (\sigma ) $ by definition of the estimator.   Thus
$$ \mu  (\exists \; i \in \ll :  \hat l_n(i) > l_i(\sigma) , {\cal B} (n, R_n) ) \le  \sum_{ \ell =1 }^{R_n} \mu   ( E_\ell ^c , {\cal B} (n,   R_n) ). $$
But 
\begin{multline*}
 E_\ell ^c \subset \\ 
 \left\{ \exists a\in \cA, \exists \, \eta \in \cA^{V_0^0 (\ell  )} : l_0 (\eta ) \le \ell  ,  \left | \hat p_n (a| \eta ) - \gamma_{\{0 \}} (a| \eta  )\right | >   \sqrt {\kappa (\delta)  \gamma_{\{0 \}} (a|\eta ) 
\frac{  
{ \log N_n (\eta )  }}{ N_n (\eta )}  }\right\} .
\end{multline*}
Hence by Theorem \ref{theo1}, for $ n \geq n_0 ,$  we have 
\begin {equation}  \label {mars4} \sum_{ \ell =1 }^{R_n} \mu  ( E_\ell ^c , {\cal B} (n,   R_n) ) \le  | \cA| C(d)R_n^{d+1}   
   \exp \left(  -c (\delta) \sqrt{     \log | \ll |    }  \right) .   \end  {equation}
By definition of $R_n,$  
\begin {equation}  \label {mars5}  R_n^{d+1}   
\le( \log \ll )^{ \frac{d+1}{2d} }.    \end  {equation}
 Taking into account \eqref {ro1a}, \eqref {ro1b}, \eqref  {mars4} and    \eqref  {mars5}    we get \eqref     {eq:badupperbound}. 
This finishes the proof of Theorem \ref{theomain1}.

\section{Proof of Theorem \ref{theomain2}}
We now turn to the problem of underestimation. We suppose that $n_0$ is sufficiently large such that $R_n \geq L .$ 
Fix $i \in \ll $ and suppose that $\hat l_n (i) < l_i(\sigma) .$ Since $l_i (\sigma ) \le L \le R_n,$   {this implies by definition 
of the estimator that for
$\ell = l_i(\sigma), $  
\begin{equation}\label{eq:aha}
\log L_n (i,\ell) \le pen (\ell ,
n) .
\end{equation}   
Recall that by (\ref{eq:loglikelihood4}),
$$ \log L_n (i,\ell) =  \left [  \sum_{j : \;  X_j^{\ell-1} = \sigma_i^{{\ell}-1}  } \frac{1}{N_n (X_j^\ell) }  \log  MPL_n (j,\ell) \right ]- \log  MPL_n (i,\ell-1). $$
By definition of $\Delta_n (\eta) $ in (\ref{eq:deltan}) we can write 
\begin{eqnarray*}
  \frac{1}{| \ll |} \log L_n (i, \ell ) & =&  
  \left(  \sum_{j : \;  X_j^{\ell-1} = \sigma_i^{{\ell}-1}  }  \frac{1}{N_n (X_j^\ell) } \Delta_n (X_j^\ell)\right)  - \Delta_n (\sigma_i^{\ell - 1}) \\
  && +  \left(  \sum_{ v \in \cA^{\partial V_0(\ell)    }  }  \sum_{a \in \cA} p (( \sigma_i^{\ell -1},v , a)) \log p (a | \sigma_i^{\ell -1}, v ) \right)
  \\
  &&- \sum_{a \in \cA} p ( ( \sigma_i^{\ell - 1 } , a )) \log p ( a | \sigma_i ^{\ell - 1 } )\\
   & =&  
  \left(  \sum_{v \in \cA^{\partial V_0 (\ell )}   }  \Delta_n (\sigma_i^{\ell -1} v)\right)  - \Delta_n (\sigma_i^{\ell - 1}) \\
  && +  \left(  \sum_{ v \in \cA^{\partial V_0(\ell)    }  }  \sum_{a \in \cA} p (( \sigma_i^{\ell -1},v , a)) \log p (a | \sigma_i^{\ell -1}, v ) \right)
  \\
  &&- \sum_{a \in \cA} p ( ( \sigma_i^{\ell - 1 } , a )) \log p ( a | \sigma_i ^{\ell - 1 } ) .
\end{eqnarray*}
Set 
\begin{eqnarray}\label{eq:delta}
D(i, \ell , \sigma  )&= &\! \! \! \!    \sum_{v \in \cA^{\partial V_0 (\ell) } }    \sum_{a \in \cA} p(( \sigma_i^{\ell -1}, v , a)) \log p (a | \sigma_i^{\ell -1}, v )\nonumber \\
 && - \sum_{a \in \cA} p ( ( \sigma_i^{\ell - 1 } , a )) \log p ( a | \sigma_i ^{\ell - 1 }) .
\end{eqnarray}
Moreover,  for a constant $t  > 0$ that will be chosen later, define 
$$ E(t  , \ell ) = \{\forall \; \eta \in \cA^{V_0^0 (\ell - 1 )}, \; \forall \; v \in \cA^{\partial V_0 (\ell)}  :  |  \Delta_n (\eta v) | \le \frac t2 \frac{1}{ | \cA |^{ | \partial V_0 (\ell ) | } } ,
\; 
 \; | \Delta_n  (\eta ) | \le t / 2 \} .$$ 
Then on $ E(t ,  \ell ) , $ 
\begin{eqnarray*}
 \frac{1}{| \ll |} \log L_n (i, \ell )& \geq&  D (i, \ell , \sigma  ) - t .
\end{eqnarray*}
Next we  show that $ D (i, \ell , \sigma  ) $ can be bounded away from zero. Taking into account
$$    \sum_{ v \in \cA^{\partial V_0(\ell) }  } p (( \sigma_i^{\ell -1},v , a)) = p ( ( \sigma_i^{\ell - 1 } , a )),$$
we can write  
$$    D (i, \ell , \sigma  ) = \! \! \! \!     \sum_{ v \in \cA^{\partial V_0(\ell) }  }  \sum_{a \in \cA} p (( \sigma_i^{\ell -1},v , a))   
\log \frac{ p (a | \sigma_i^{\ell -1},v )  }{ p ( a | \sigma_i ^{\ell - 1 }) }.$$
By  Pinsker's inequality for relative entropy (see for example Fedotov et al. (2003), \cite{fedotov}), we have that 
for $P$ and $Q$   probability distributions on $\cA,$  
$$ \sum_a P(a) \log \frac{P(a)}{Q(a)} \geq \frac12 || P- Q ||^2_{TV} .$$
%This is exactly Pinsker's inequality. 
Moreover,
$$ || P- Q ||^2_{TV} \geq \sup_a (P(a) - Q(a))^2 .$$    
But,  since $ \ell =  l_i(\sigma),$  there exist $v$ and $a$ such that $  p  (a | \sigma_i^{\ell - 1 }) \neq p   (a | \sigma_i^{\ell -1  },v ) .$
Hence we have that $  D (i, \ell , \sigma  ) > 0 .$
Since we are working under the assumption  that $l_i (\omega ) \le  L $ for all $\omega$ (recall condition (\ref{eq:fin}) of Assumption \ref{ass:mixing}), we can thus conclude that 
\begin{equation}\label{eq:lowerbounddelta}
D (i, \ell , \sigma ) \ge  d_0 > 0,
\end{equation}
where 
 $$ d_0 = \inf_{i, \sigma} D (i, l_i (\sigma )  , \sigma ) > 0 .$$ 
Choosing now $t   = \frac {d_0}   2 ,$ we finally obtain that on $E( \frac{d_0}{2}, \ell),$ for $\ell = l_i (\sigma ) ,$ 
$$ \log L_n ( i, \ell ) - pen (\ell ,n) 
\geq | \ll | \frac {d_0} { 2} - pen ( \ell , n) > 0  $$
for $n \geq n_0 ( i), $ since $pen( \ell ,n) =   \kappa |\cA|  |\cA|^{|\partial V_0(\ell)|   }  \log | \Lambda_n| = O (  \log | \Lambda_n| ) . $ 
This is in contradiction to (\ref{eq:aha}) and implies that $\hat l_n (i) \geq l_i ( \sigma ) . $ Hence we conclude that 
\begin{equation} \label{mars7}
 \mu [\exists i \in \Lambda_n:  \hat l_n (i) <  l_i(\sigma) ]  \le     \sum_{ \ell \le L } \mu  \left [  E \left( \frac {d_0} {2}, \ell\right)^c  \right]. 
\end{equation}
We use (\ref{eq:deltan2}) and sum over all possibilities of choosing $\eta \in \cA^{V_0^0 (\ell - 1)} $ and of choosing patterns $ X_j^\ell $ such that $X_j^{\ell - 1 } = \eta, $ which gives $|\cA |^{ | V_0 (\ell )  |} $ terms. But since for $\ell \le L ,$ $ |\cA |^{ | V_0 (\ell )  |} \le |\cA |^{ (2L+1)^d },$ we finally obtain 
\begin{eqnarray*}
&&  \mu \left( \left( E (\frac {d_0} {2} , \ell )\right)^c  \right) \\
&& \le 3  |\cA | e^{ 1/e} \left( |\cA|^{(2L+1)^d}   \right) 
\exp \left(  - C(d,L) \frac{|\ll |     {d_0}   \alpha^2_0 }{ 8  |\cA |^2  [\log^2 \alpha_0 ] |\cA |^{ | \partial V_0 (\ell )  |}   e } \right) \\
&& \le 3 |\cA | e^{ 1/e} \left(|\cA |^{ (2L+1)^d }   \right) 
\exp \left(  - \frac{C(d,L) |\ll |   {d_0}   \alpha^2_0   }{ 8 |\cA |^2  [\log^2 \alpha_0 ] |\cA |^{ | \partial V_0 (L) | }  e } \right)\\
&& = 3 |\cA | e^{ 1/e} \left(|\cA |^{ (2L+1)^d }   \right) 
\exp \left(  - C_2 (d,L, q_{min}) |\ll |   \right)
\end{eqnarray*}
(recall the control of $\alpha_0$ given in (\ref{LE.1})), where $C_2 (d,L, q_{min}) $ is another constant depending on the dimension $d ,$ on the interaction range $L$ and on $q_{min} .$  

Thus, we can conclude that  for any   $0 < \e < 1 $  there exists $n_0 = n_0 (\e, q_{min}, L,d )$ such that for all $ n \geq n_0,$ 
\begin{equation}\label{mars6}
 \sum_{ \ell \le L } \mu  \left(  E \left( \frac {d_0} {2}, \ell\right)^c  \right) \le      \exp \left( -  | \ll |^{1 - \e } \right). 
\end{equation}
\eqref {mars7} and \eqref {mars6} together conclude the proof of Theorem \ref{theomain2}.  $\bullet$

\section{Final comments} 
  We generalize the concept of chains with memory of variable length to the
multidimensional case of random fields. The main aim of this concept is to 
adopt a parsimonious way of describing data: the symbol at site $i$ is influenced  only by        a random set of symbols, the set depends     on the  observed   data.  
As in the case of one dimensional models, the set of relevant neighbor states of site $i$ is called
the {\it context}.   The radius of the smallest ball containing the support of the context  is  the {\it length of the context} of site $i.$

We presented in Section \ref{section3} an estimator of the context length function based on a sequence of local
decisions between two possible context lengths. These decisions are performed using the log likelihood ratio function. 
In the case of dimension one, our estimator is simply the context length estimator of 
variable length chains which has been classically considered in the literature.  We refer the interested reader
to Galves and L\"ocherbach (2008), \cite{GL}, for a survey and bibliographic comments.

\section{Appendix}
At the end of Section \ref{notation} we argued that in order to estimate the context of a finite set of sites $\Lambda 
\Subset \Z^d$ it is sufficient to estimate the contexts of the one-point specification. In particular, we
stated formula (\ref{eq:freiburg1}), which relates $c_\Lambda (\omega )$ to $c_i (\omega). $ 
%This is a consequence
%of results of Georgii (1988), \cite{geo88}, guaranteeing that the specification is uniquely determined by the single-site 
%conditional probabilities. 
In the first subsection of this appendix we show 
this.   % we adapt Theorem 1.33 of Georgii (1988), \cite{geo88}, 
%to our situation and prove formula (\ref{eq:freiburg1}). 
In the second subsection, we complete  the computations 
%needed to define 
%the one site specification
 for the example %\ref {ex:ex1} and  
 \ref{ex:ex2}.
Finally we state a deviation 
inequality needed in  Section \ref{section:5}.
  
\subsection{From one point specifications to several points}
It is well known  in Statistical Mechanics that the positive one point specification uniquely determines the family of specifications, see Theorem 1.33 of Georgii (1988), \cite{geo88}.
This result still holds for   Variable-neighborhood   random fields, 
since they can be embedded   into     classical  random fields.  
But we would like to  determine if and how the context of one single site  determines the  $ \Lambda-$contexts of the  specification, for any $ \Lambda \Subset \Z^d$.  Proposition \ref  {p1}  gives an answer.

We consider
   local  specifications   $\gamma$ which are
 positive, i.e. 
 $$ \gamma_\L ( \om_\L| \cdot ) >0,\quad   \hbox {for all }  \om_\L \in \cA^\L \quad   \hbox {and }  \quad \L \Subset \Z^d. $$
    In the following it will be convenient to write
$$ \gamma_\L (\{ \omega_\L\} | \sigma ) = \varrho_\L ( \omega_\L
  \sigma).$$
This family  $ \{ \varrho_\L, \L \Subset \Z^d\}$ is a family of functions 
$ \varrho_\L : \Omega \to [0, 1] $ satisfying the following two conditions: 
   \begin{equation}
    \sum_{\omega_\Lambda\in \cA^{\Lambda}} \rho_\L (\om_\Lambda \sigma)
      \;=\; 1 , \quad \forall \sigma \in \Omega,
    \label{eq:5a}
  \end{equation}
  and for every $\Lambda\subset\Delta\Subset\lat$, all $\omega, \eta, \sigma$   in 
  $\Omega $  we have
  \begin{equation}
  \frac { \rho_\Delta(\omega_\Lambda \sigma_{\Delta \setminus \Lambda} \sigma_{\Delta^c})}  {
  \rho_\Delta(\eta_\Lambda \sigma_{\Delta \setminus \Lambda} \sigma_{\Delta^c})}
  =   \frac { \rho_\Lambda (\omega_\Lambda   \sigma_{\Lambda^c})}  {
   \rho_\Lambda (\eta_\Lambda   \sigma_{\Lambda^c})}.
    \label{eq:cons}
  \end{equation}

    \noindent
    \begin {prop} \label {p1}   Assume that the family of  local specifications
      $ \gamma$ defined in \eqref {s2} is positive  \footnote {The positivity requirement can be relaxed,  under some minor modifications of  the  proof, see Georgii (1988), \cite{geo88},  Theorem 1.33. }.
      % and denote by
%$$ \rho_{\{i\}}(\omega)= \gamma_{i } (\sigma(i)= \omega(i) |  c_{i} %(\omega)), \quad   i \in \Z^d.$$ 
We have the    following:
\begin {itemize}
\item  $\gamma$ is uniquely determined by $ \{  \gamma_{\{i\}} (\cdot| c_i(\om)), \,i \in \Z^d, \om \in \Omega\}$.
%$ \{  \rho_{\{i\}}, \,i \in \Z^d\}$.
\item  For $ \Lambda \Subset \Z^d$,
\begin{equation}\label{eq:freiburg1bis}
  \supt_{\Lambda} (\omega)=  \cup_{\omega_\Lambda} \left ( \cup_{i \in \Lambda} \supt_{\{i\}}(\omega)\right) \setminus \Lambda.   
\end{equation}  
 \end {itemize} 
\end {prop} 

\begin{proof}
% Since $ \cA$ is countable and $\lambda$ is the counting measure  %there exists a  family  of measurable functions from $ \Omega \to [0,%\infty)$,  $ \rho:= \{ \rho_\L, \; \L \subset \Z^d\} $  so that $%\gamma=\rho \lambda$. 
  Recall that we set  $\gamma_{\{i\}} (\{\om (i)\}| c_i(\om)) =   \rho_{\{i\}} (\om)$.   Further  $ \rho_{\{i\}}(\omega)>0$ for $\om \in \Omega$ since we assumed that $\gamma$ is   positive. For each fixed $ \omega(i),$   $ \omega_{ \{i\}^c } \mapsto \rho_{\{i\}} ( \omega) $ is  a measurable function with respect to   $ \FF_ {\supt_{\{i \} } }$, see \eqref {F1}. 
  For each $\L$, Georgii (1988), \cite{geo88}, shows in the proof of Theorem 1.33 how to determine $ \rho_\L$ in terms  of $\{ \rho_{\{i\}}, i \in \Z^d\}$ such that    for any  measurable function $f$ we have that
  \begin{equation}\label{paris3} \int f(\om) d \mu (\om)=    \int  d \mu (\om_{\L^c}) \sum_{ \om_{\L}\in \cA^{\L} }    f(\om) \rho_\L (\om) , \quad \forall \L \Subset \Z^d, \end {equation}
  where $ \mu$ is any measure on $ \Omega$ so that
$$ \int  f(\om)   d \mu (\om)=  \int d \mu (\om_{\{i\}^c})  \sum_{ \om_{\{i\} }\in \cA}    f(\om) \rho_{i} (\om).  $$
This immediately  shows that    $ \g$ is uniquely determined by $\rho_{\{i \}}$.  
To construct $ \rho_\Lambda$ and to prove (\ref{eq:freiburg1}), one proceeds by induction on $| \Lambda | .$ The case   $|\Lambda|=1$ is trivial. 
Suppose then that $ \rho_{\Lambda_1}$ and $ \rho_{\Lambda_2}$ have been constructed.  Let $ \Lambda$ be the union of two disjoint sets, $ \Lambda = \Lambda_1 \cup \Lambda_2$, $ \Lambda_1 \cap \Lambda_2 = \emptyset$.  Define   \begin{equation}\label{eq:freiburg2}
  \rho_\Lambda (\omega)   = \frac { \rho_{\Lambda_1}(\omega)  } {\sum_{\bar \omega_{\Lambda_1}}    \frac    { \rho_{\Lambda_1}(\bar \omega_{\Lambda_1}  \omega_{\Lambda_1^c} )  }   { \rho_{\Lambda_2}(\bar \omega_{\Lambda_1}  \omega_{\Lambda_1^c})  }  }. 
%= \frac { \rho_{\Lambda_1}(\omega)  } { \int  \lambda_{\Lambda_1} (d \omega'  | %\omega  )  \frac    { \rho_{\Lambda_1}(\omega ')  }   { \rho_{\Lambda_2}(\omega ')  }  } . 
\end{equation} 
By induction, for any given $ \omega_{ \Lambda_1}, $ $ \omega_{ \Lambda_1^c } \mapsto \rho_{\Lambda_1}(\omega)= 
\rho_{\Lambda_1}(\omega_{ \Lambda_1},  \omega_{ \Lambda^c_1})$ depends
only on $ \supt_{\Lambda_1} (\omega) ,$ and for any given $ \omega_{ \Lambda_2}, $ $ \omega_{ \Lambda_2^c } \mapsto \rho_{\Lambda_2}(\omega_{ \Lambda_2},  \omega_{ \Lambda^c_2})$ only on $ \supt_{\Lambda_2} (\omega) .$ Hence (\ref{eq:freiburg2}) implies that 
for any given $\omega_\L$, the function  $\omega_{ \Lambda} \mapsto \rho_\Lambda (\omega)$ depends, by construction, on the $\sigma -$algebra generated by $\supt_{\Lambda_1} (\omega)  \cup 
\bigcup_{ \bar \om_{\Lambda_1}} \supt_{\Lambda_2} (\bar \om_{\Lambda_1}\omega_{\L_1^c})$.  
Note that in general     $\supt_{\Lambda_1} (\omega)\cap \supt_{\Lambda_2} (\omega) \neq \emptyset$.   Therefore the value of $\omega_{\Lambda_1}$  might be relevant for determining   $\supt_{\Lambda_2} (\omega)$ and the  value of $\omega_{\Lambda_2}$  might be relevant for determining   $\supt_{\Lambda_1} (\omega)$.   To   have a function 
$\rho_\Lambda (\omega)$ measurable for any choice of $ \omega_\Lambda$ we  set
 $$  \supt_{\Lambda} (\omega)=  \left [ \cup_{\omega_{\Lambda_1}}   \cup_{\omega_{\Lambda_2}} \left (\supt_{\Lambda_1} (\omega)\cup \supt_{\Lambda_2} (\omega) \right) \right ] \setminus  (\Lambda_1 \cup \Lambda_2).    $$
In this way, for any choice of  $\omega_\L,$ $  \rho_\Lambda (\omega) $  is   $ \FF_{ \supt_{\Lambda}} -$measurable.  
 It is immediate to verify by induction that one  has 
  $$  \supt_{\Lambda} (\omega)= \left [  \cup_{\omega_\Lambda} \cup_{i \in \Lambda} \supt_{\{i\}}(\omega)  \right ]  \setminus \Lambda.    $$
  We need to show that \eqref {paris3} holds. By induction, taking in account that $ \om = (\om_{\L_k}, \om_{\L_k^c}),$
  \begin{equation}\label{paris4}  \int f(\om) d \mu (\om)=    \int  d \mu (\om_{\L _k^c}) \sum_{  \om_{\L_k}}    f(\om) \rho_{\L _k}( \om),  \quad k=1,2  \end {equation}
holds.   To show that this holds for $\rho_\Lambda  $ take a positive  measurable function $f$ defined on $\Omega$.  We have 
  \begin{eqnarray*} 
&& \int d \mu (\bar \om) \sum_{\om_\L}  f(\om_\L \bar \om_{\L^c} )     \\
&& = \int d  \mu (\bar \om)  \sum_{\om_{\L_2} } \sum_{\om_{\L_1} }  f(\om_{\L_1}\om_{\L_2} \bar \om_{\L^c}  ) 
\\
&&  =\int d  \mu (\bar \om)   \sum_{\om_{\L_2} }  \rho_{\L _2}(\om_{\L_2}\bar \om_{\L_2^c}) \rho^{-1}_{\L _2}(\om_{\L_2}\bar \om_{\L_2^c})   \sum_{\om_{\L_1} }  f(\om_{\L_1}\om_{\L_2} \bar \om_{\L^c}  ).
\end  {eqnarray*}  
But applying \eqref{paris4} first to $\L_2, $ then to $\L_1$, this last line can be written as
 \begin{eqnarray*} 
&&   \int d  \mu (\bar \om)   \sum_{\om_{\L_2} }  \rho_{\L _2}(\om_{\L_2}\bar \om_{\L_2^c}) \rho^{-1}_{\L _2}(\om_{\L_2}\bar \om_{\L_2^c})   \sum_{\om_{\L_1} }  f(\om_{\L_1}\om_{\L_2} \bar \om_{\L^c}  ) 
 \\
&& =  \int d  \mu (\bar \om) \,   \rho^{-1}_{\L _2}( \bar \om)    \left [ \sum_{\om_{\L_1} }  f(\om_{\L_1} \bar \om_{\L_1^c}  ) \right ] 
   \\
&& =    \int d  \mu (\bar \om)  \sum_{ \tilde  \om_{\L_1}} \rho_{\L_1} (  \tilde  \om_{\L_1} \bar  \om_{\L_1^c}  )  \rho^{-1}_{\L _2}( \tilde  \om_{\L_1} \bar  \om_{\L_1^c})   \left [ \sum_{\om_{\L_1} }  f(\om_{\L_1} \bar \om_{\L_1^c}  ) \right ]  \\
%&& = \int d \mu (\om ) \int \lambda_{\L_1} (d \om' | \om ) f(\om ') \left[ \int %\lambda_{\L_1} (d \om' | \om ) \rho_{\L_1} ( \om ')  \rho^{-1}_{\L _2}(\om '  ) \right] \\
&& =  \int d  \mu (\bar \om)   \sum_{\om_{\L_1} }  f(\om_{\L_1} \bar \om_{\L_1^c}  )   \rho_{\L_1} (\om_{\L_1} \bar \om_{\L_1^c}  )  \rho^{-1}_{\L_1} (\om_{\L_1} \bar \om_{\L_1^c}  )  
 \left [\sum_{ \tilde  \om_{\L_1}} \rho_{\L_1} (  \tilde  \om_{\L_1} \bar  \om_{\L_1^c}  )  \rho^{-1}_{\L _2}( \tilde  \om_{\L_1} \bar  \om_{\L_1^c})  \right ]
\\ 
&& =  \int d  \mu (\bar \om)   \sum_{\om_{\L_1} }  f(\om_{\L_1} \bar \om_{\L_1^c}  )  \rho_{\L_1} ( \tilde  \om_{\L_1} \bar  \om_{\L_1^c}) 
  \rho_\L^{-1}( \tilde  \om_{\L_1} \bar  \om_{\L_1^c}) \\&&
  = \int d  \mu (  \om_{\L_1^c})   \sum_{\om_{\L_1} }  f(\om)  \rho_{\L_1} (\om) 
  \rho_\L^{-1}(\om) .      
\end  {eqnarray*}  
Applying once more \eqref{paris4}, 
we obtain 
$$ 
    \int d  \mu (  \om_{\L_1^c})   \sum_{\om_{\L_1} }  f(\om)  \rho_{\L_1} (\om) 
  \rho_\L^{-1}(\om)  
     = \int d \mu ( \om) f (\om )  \rho_\L^{-1} (\om ).      
$$
By applying the above equality to $ f(\om) \rho_\L (\om ) $ instead of $f (\om) $  we get  the result.    The above definition of $\rho_\L$ depends on the choice of $\L_1$ and $\L_2$; one needs to obtain an unambiguous definition of $\rho_\L$ to choose a definite strategy to exhausting $\L$ site by site. 
\end{proof}

\subsection{Continuation of the example %\ref{ex:ex1} and 
\ref{ex:ex2}}

\noindent
{\bf Continuation of example \ref{ex:ex2}}\\
We prove formula (\ref{eq:oho}), using   (\ref{eq:oho2}). First note that $  i \in \Gamma_j (\omega)  $ implies that $ \| i - j \| \le L .$ Now 
if $i \in \Gamma_j (\omega ), $ we have two cases. Either $i \in \hat \Gamma^1_i ( \omega ) ,$
in which case $\Gamma_i^1 (\omega ) = \Gamma_j^1 (\omega) .$  Or $ i \in \partial \Gamma_j^1 (\omega ) .$ Then $\omega(i)i = 1,$ and in this case,
$ i \in \hat \Gamma_j (\omega^i) .$ Then the same arguments as above show that 
$$ \Gamma^1_j (\omega^i ) =  \Gamma_i^1 (\omega )  .$$
Hence 
$$  \left[ \bigcup_{j \in \Z^2} \{ \Gamma^1_j (\omega ) : i \in \Gamma_j (\omega) \}
  \cup \bigcup_{j \in \Z^2} \{ \Gamma^1_j (\omega^i ) : i \in \Gamma_j (\omega^i) \}
\right] =  \Gamma_i^1 (\omega )  .$$
Finally, by definition of $ \Gamma_j (\omega), $ 
\begin{multline*}
 \left[ \bigcup_{j \in \Z^2} \{ \Gamma_j (\omega ) : i \in \Gamma_j (\omega) \}
  \cup \bigcup_{j \in \Z^2} \{ \Gamma_j (\omega^i ) : i \in \Gamma_j (\omega^i) \}
\right] \\
= \left[ \bigcup_{j \in \Z^2} \{ \Gamma^1_j (\omega ) : i \in \Gamma_j (\omega) \}
  \cup \bigcup_{j \in \Z^2} \{ \Gamma^1_j (\omega^i ) : i \in \Gamma_j (\omega^i) \}
\right]  \cap  \bigcup_{ j \in V_i (L) } V_j (L) \\
= \Gamma_i^1 (\omega) \cap V_i(2L) .
\end{multline*} 
This concludes the proof.
{\hfill $\bullet$ \vspace{0.25cm}}

We close with the following version of Bernstein's inequality obtained by Friedman (1975),  for discrete-time martingales having bounded jumps, see for instance Dzhaparidze and van Zanten\cite{vanzanten}.

\noindent
\begin{lemma}\label{le:bernstein}
Let $M_n= \xi_1 + \ldots + \xi_n$ be a discrete martingale with respect to some filtration $({\cal F}_n)_{n \ge 0} $ having bounded jumps $|\xi_n| \le a .$ Let 
$$ <M>_n = \sum_{i = 1}^n E ( \xi_i^2 | {\cal F}_{i-1})   .$$
Then
$$ P ( \max_{k \le n } |M_k| > z ; <M>_n \le L ) \le 2 \exp \left( - \frac12 \left( \frac{z^2}{L} + \frac{a z }{3} \right) \right) .$$
 \end{lemma}

Eva L\"ocherbach

CNRS UMR 8088

D\'epartement de Math\'ematiques

Universit\'e de Cergy-Pontoise

95 000 CERGY-PONTOISE,  France

email: {\tt eva.loecherbach@u-cergy.fr}

\bigskip

  Enza Orlandi
 
Dipartimento di Matematica

Universit\`a  di Roma Tre

 L.go S.Murialdo 1, 00146 Roma,  Italy. 

email: {\tt orlandi@mat.uniroma3.it}

\begin{thebibliography}{99}
  

\bibitem{Az} Azencott, R.
{\it Image analysis and Markov fields. }
ICIAM 87: Proceedings of the first international conference on industrial and applied mathematics (Paris 1987),  SIAM Philadelphia, PA.  53-61 (1987).

\bibitem{Be} Besag, J.
{\it Spatial interaction and the statistical analysis of lattice systems.}
J. Ro.  Stat. Soc. Ser. B 36,   192-236 (1974).

%\bibitem{chazotteetal}  Chazottes, J.-R., Collet, P., K\"ulske, C., Redig, F.
%{\it Concentration inequalities for random fields via coupling.} 
%Probab. Theory Relat. Fields 137, No. 1-2, 201-225 (2007).

\bibitem{comets} Comets, F. {\it On Consistency of a Class of Estimators for Exponential Families of Markov Random Fields on the Lattice.} Ann. Statist. 20, No. 1, 455-468 (1992).

\bibitem{CT} Csisz\`ar, I.,Talata, Z.
{\it Consistent estimation of the basic neighborhood of Markov random fields. }
Ann. Statist. 34, No. 1, 123-145 (2006).

\bibitem{CTbis} Csisz\`ar, I.,Talata, Z. 
{\it Context tree estimation for not necessarily finite memory
processes, via BIC and MDL.}
IEEE Trans. Inform. Theory 52, 1007-1016 (2006).

\bibitem{dedecker} Dedecker, J. %\'er\^ome.
{\it Exponential inequalities and functional central limit theorems for a random field.} ESAIM, Probab. Stat. 5, 77-104 (2001). 

\bibitem{DDG} Dereudre, D., Drouilhet, R., Georgii, H.O.
{\it Existence of Gibbsian point processes with geometry-dependent interactions.}  
arxiv.org/abs/1003.2875 (2010).
 
\bibitem{DL} Dereudre, D., Lavancier, F.
{\it Practical simulation and estimation for Gibbs Delaunay-Voronoi tessellations with geometric hardcors interaction.}
Comp. Statistics and Data Analysis 55, 498-519 (2011). 

 \bibitem{D1}  Dobrushin, R. L.,    {\it   Gibbsian random fields for lattice systems with pairwise interactions.}  Funct. Anal. Appl. 2, 292-301,  (1968).
 
\bibitem{D2}  Dobrushin, R. L.,    {\it  The problem of uniqueness of  Gibbs  random field  and the problem of phase transition.}  Funct. Anal. Appl. 2, 302-312,  (1968).
 
\bibitem{vanzanten} Dzhaparidze, K., van Zanten, J.H. {\it On Bernstein-type inequalities for martingales.} Stochastic Processes Appl. { 93}, 109-117 (2001).

\bibitem{EL} Efros, A.A., Leung, T.K. 
{\it Texture Synthesis by Non-parametric sampling.} IEEE International Conference on Computer Vision, Corfu, Greece, September 1999 (1999). 


\bibitem{fedotov} Fedotov, A.,  Harremo{�e}s, P. and Tops{\o}e, F. {\it Refinements of Pinsker's Inequality.} IEEE Trans. 
Inf. Theory, vol. 49, pp. 1491--1498 (2003).

\bibitem{ferrariwyner} Ferrari, F.F., Wyner, A. {\it  Estimation of general stationary processes by variable length Markov Chains.} Scand. J. Stat. 30, No. 3, 459--480 (2003).

\bibitem{finesso} Finesso L., Liu C.C., Narayan P. {\it The optimal error exponent for Markov order estimation.} IEEE Trans. Inf. Theory, vol. 42, pp. 1488--1497 (1996).

 \bibitem{galvesleonardi} Galves, A., Leonardi, F.
{\it Exponential inequalities for empirical unbounded context trees. }
Sidoravicius, Vladas (ed.) et al., In and out of equilibrium 2. Papers celebrating the 10th edition of the Brazilian school of probability (EBP), Rio de Janiero, Brazil, July 30 to August 4, 2006. Basel: Birkh\"auser. Progress in Probability 60, 257-269 (2008). 

\bibitem{GL} Galves, A.,   L\"ocherbach, E.   {\it  Stochastic chains with memory of variable length.} ``Festschrift in honour of the 75th birthday of Jorma Rissanen", Tampere University Press, 2008, arxiv.org/abs/0804.2050.   


\bibitem{GMS} Galves, A.,   Maume-Deschamps,  V.,   Schmitt, B   {\it   Exponential inequalities for VLMC empirical trees.}   
  ESAIM Prob. Stat., vol 12, pp.219-229 (2008).
 
 
\bibitem{GOT} Galves, A.,   Orlandi, E.,   Takahashi, D.Y   {\it   Identifying interacting pairs of sites in infinite range Ising models.}   
 arxiv.org/abs/1006.0272   

\bibitem{geo88} Georgii, Hans-Otto
{\it Gibbs measures and phase transitions}. De Gruyter Studies in Mathematics, 9. Berlin etc.: Walter de Gruyter. (1988). 

\bibitem {gidas} Gidas, B. 
{\it  Parameter estimation for Gibbs distributions from fully observed data} In {\it Markov Random Fields: Theory and Application }  471-498. Academic Press, Boston 

\bibitem{Gido} Gidofalvi, G. 
{\it Texture Synthesis by
Non-parametric Sampling / Image Quilting
for Texture Synthesis $\&$ Transfer} Seminar CSE 291: Seminar on Vision $\&$ Learning, Department of Computer Science and Engineering, Universy of California, San Diego, cseweb.ucsd.edu/classes/fa01/cse291/TextureSynthesis.ppt .

\bibitem{Gr}Grunwald P. D. {\it  The minimum description length principle.}  MIT Press (2007). 


\bibitem{JS} Ji, C.,  Seymour, L. 
{\it A consistent  model selection procedure for  Markov random fields based on penalized pseudolikelihood.}  
Ann. App. Probab. 6,  423-443 (1996).

\bibitem{merhav} Merhav, N., Gutman, M., Ziv, J. {\it On the estimation of the order of a Markov chain and universal data compression.} 
IEEE Trans. Inf. Theory, vol. 35, pp. 1014--1019 (1989).


\bibitem{P} Presutti, E. {\it Scaling Limits in Statistical Mechanics and Microstructures in Continuum Mechanics.}
Springer Berlin Heidelberg, Series: Theoretical and Mathematical Physics,
(2009).

\bibitem{R} Rissanen, J. {\it   A  universal data compression system.}
IEEE,  Trans. Inform. Theory,{ 29}, 656-664 
(1983).


\end{thebibliography}
 \end{document}